\documentclass[11pt]{article}
\pdfoutput=1
\usepackage[authoryear]{natbib}
\usepackage{amsmath,amssymb,amsfonts,amsthm,bbm}
\usepackage{epic,eepic,epsfig,longtable}
\usepackage{multirow,verbatim}
\usepackage{array}
\usepackage{verbatim}
\usepackage{epsfig}
\usepackage{setspace}
\usepackage{float}
\usepackage{CJK}
\usepackage{color}
\usepackage{algorithm}
\makeatletter

\makeatletter
\def\singlespace{\def\baselinestretch{1}\@normalsize}

\numberwithin{equation}{section}

\renewcommand{\hat}{\widehat}

          \def\cL{{\cal  L}}     
               
          \def\cN{{\cal  N}}

\renewcommand{\hat}{\widehat}

\def \heps     {\hat{\heps}}

\DeclareMathOperator*{\argmin}{argmin}

\DeclareMathOperator{\sgn}{sgn}

\DeclareMathOperator{\Var}{Var}
\DeclareMathOperator{\var}{var}

\def \var   {\mbox{var}}

\def \sgn   {\mbox{sgn}}

\def \eff {{\rm{eff}}}

\def\EE{\mathbb{E}}
\def\PP{\mathbb{P}}

 at 8truept

\def\today{\ifcase\month\or
  January\or February\or March\or April\or May\or June\or
  July\or August\or September\or October\or November\or December\fi
  \space\number\day, \number\year}

\def \newpage {\vfill\eject}

\newdimen\biblioindent\biblioindent=30pt
\newcommand{\beq}  {\begin{equation}}
\newcommand{\eeq}  {\end{equation}}
\newcommand{\beqn} {\begin{eqnarray}}
\newcommand{\eeqn} {\end{eqnarray}}
\newcommand{\beqnn}{\begin{eqnarray*}}
\newcommand{\eeqnn}{\end{eqnarray*}}

\allowdisplaybreaks
\renewcommand{\baselinestretch}{1.5}
\newtheorem{lem}{Lemma}

\newtheorem{thm}{Theorem}
\newtheorem{rem}{Remark}
\newtheorem{assumption}{Assumption}
\newcounter{CondCounter}
\def \RR	{\mathbb{R}}

\def \mse {\mathrm{MSE}}
\def \ql {\mathrm{QL}}
\newcommand{\rnum}[1]{\uppercase\expandafter{\romannumeral #1\relax}}

\begin{document}

	\title{\vspace*{-0.5 in} Volatility prediction comparison via robust volatility proxies: An empirical deviation perspective}
	\author{Weichen Wang\thanks{Address: Innovation and Information Management, HKU Business School, The Univsersity of Hong Kong, Hong Kong. E-mail: \textit{weichenw@hku.hk}.}$\;$, Ran An$^{\dagger}$, Ziwei Zhu\thanks{Address: Department of Statistics, University of Michigan, Ann Arbor, MI 48109, USA. E-mail: \textit{ziweiz@umich.edu}, \textit{angrace@umich.edu}. The research was partially supported by NSF grant DMS-2015366.} 
	\medskip\\
    \normalsize
    $^*$Innovation and Information Management, HKU Business School\\
    \normalsize
    $^{\dagger}$Department of Statistics, University of Michigan
    \date{}}

	\maketitle
	
	\begin{abstract}
	Volatility forecasting is crucial to risk management and portfolio construction. One particular challenge of assessing volatility forecasts is how to construct a robust proxy for the unknown true volatility. 
% 	\cite{patton2011volatility} studies the loss selection problem and suggested to choose a robust loss defined in his paper, with the mean-squared error and the quasi-likelihood loss as two commonly used examples. The advantage of robust losses is that for any unbiased proxy, the ranking of two predictors will be always consistent in terms of the long-run expectation. However, in practice we have to use the empirical loss rather the expected loss for evaluation.
	In this work, we show that the empirical loss comparison between two volatility predictors hinges on the deviation of the volatility proxy from the true volatility. We then establish non-asymptotic deviation bounds for three robust volatility proxies, two of which are based on clipped data, and the third of which is based on exponentially weighted Huber loss minimization. In particular, in order for the Huber approach to adapt to non-stationary financial returns, we propose to solve a tuning-free weighted Huber loss minimization problem to jointly estimate the volatility and the optimal robustification parameter at each time point. We then inflate this robustification parameter and use it to update the volatility proxy to achieve optimal balance between the bias and variance of the global empirical loss. We also extend this Huber method to construct volatility predictors. Finally, we exploit the proposed robust volatility proxy to compare different volatility predictors on the Bitcoin market data. It turns out that when the sample size is limited, applying the robust volatility proxy gives more consistent and stable evaluation of volatility forecasts. 
	\end{abstract}
	
	
\newpage	
\section{Introduction}
Volatility forecasting is a central task for financial practitioners, who care to understand the risk levels of their financial instruments or portfolios. There have been countless researches on improving the volatility modeling for financial time series, including the famous ARCH/GARCH model for better modeling volatility clustering, its many variants and more general stochastic volatility models \citep{engle1982autoregressive, bollerslev1986generalized, baillie1996fractionally, taylor1994modeling}, and on proposing better volatility predictors under different model settings and objectives \citep{poon2003forecasting, brailsford1996evaluation, andersen2005volatility, brooks2003volatility, christoffersen2000relevant}. This list of volatility forecasting literature is only illustrative and far from complete for the large body of researches on this topic. 

The prediction ideas range from the simplest Exponentially Weighted Moving Average (EWMA) \citep{taylor2004volatility}, which is adopted by J. P. Morgan's RiskMetrics, to more complicated time series models and volatility models including GARCH \citep{brandt2006volatility, park2002outlier}, to option-based or macro-based volatility forecasting \citep{lamoureux1993forecasting, vasilellis1996forecasting, christiansen2012comprehensive}, and to the more advanced machine learning techniques such as the nearest neighbor truncation \citep{andersen2012jump} and Recurrent Neural Netowrks (RNN) \citep{guo2016robust}. Correspondingly, the underlying model assumption ranges from only smoothness of nearby volatilities, to different versions of GARCH, to Black-Scholes model \citep{black2019pricing} and its complicated extensions. The data distribution assumption can also vary in whether data are normally distributed, or heavy-tailed distributed from a known distribution e.g. t-distribution, or generally non-normal. When data are generally non-normal, researchers have proposed to use the quasi maximum likelihood estimation (QMLE) \citep{bollerslev1992quasi, charles2019volatility, carnero2012estimating} and its robust standard error for inference, but the theoretical results are typically asymptotic. Albeit good theoretical guarantee, industry practitioners seldom apply QMLE and tend to employ the naive approach of truncating the returns by an ad-hoc level and then applying EWMA. 

In this work, we consider a model assumption requiring only smoothness of volatilities. For simplicity, we also assume the volatility time series are given a-priori, and after conditioning on the volatilities, return innovations are independent. 
We choose this simple setting for the following reasons. Firstly, our main focus of study is on building effective robust proxies rather than testing volatility models and constructing fancy volatility predictors. 
Secondly, although we ignore the weak dependency between return innovations (think of ARMA models \citep{brockwell2009time} for weakly dependency), the EWMA predictors and proxies can still have strong temporal dependency, due to data overlapping of a rolling window, so our analysis is still nontrivial. Also note that we allow the return time series to be non-stationary. Thirdly, the motivating example for us is volatility forecasting for the crypto market. \cite{charles2019volatility} applied several versions of GARCH models characterized by short memory, asymmetric effect, or long-run and short-run movements and concluded that they all seem not appropriate to model Bitcoin returns. Therefore, starting from conditionally independent data without imposing a too detailed model e.g. GARCH may be a good general starting point for the study of robust proxies. 

Besides the native EWMA predictor as our comparison benchmark, we consider a type of robust volatility predictor when the instrument returns present heavy tails in their distributions. Specifically, we only require the returns to bear finite fourth moment. We consider the weighted Huber loss minimization, which turns out to be a nontrivial extension from the equal-weighted Huber loss minimization. To achieve the desired rate of convergence, the optimal Huber truncation level for each sample should also depend on sample weight.
In addition, we apply a tuning-free approach following \cite{wang2020new} to tune the Huber truncation level adaptively and automatically. Unlike QMLE, our results focusing on a non-asymptotic empirical deviation bound. %In addition, the Huber truncation in some sense resembles what has been commonly done in the financial industry.
Therefore, although the main contribution of the paper is on robust proxy construction, we also claim a separate contribution on applying Huber minimization in the EWMA fashion.

Now, given two volatility predictors, the evaluation of their performance is often quite challenging due to two things: (1) selection of loss functions, and (2) selection of proxies, since obviously we cannot observe the truth volatilities. The selection of loss functions have been studied by \cite{patton2011volatility}. In Patton's insightful paper, he defined a class of robust losses with the ideal property that for any unbiased proxy, the ranking of two predictors using one of the robust losses will be always consistent in terms of the long-run expectation. The property is desired for it tells risk managers to select a robust loss, then not to worry much on designing proxies. As long as the proxy is unbiased, everything should just work out. 
Commonly used robust losses include the mean-squared error (MSE) and the quasi-likelihood loss (QL). However, there is one weakness of Patton's approach which has not been emphasized much in previous literature: the evaluation has to be in long-run expectation. The deviation of the empirical loss, which is what people actually use in practice, from the expected loss may still cause a large variance due to a bad choice of volatility proxy. 
Put it in the other way, his theory did not tell the risk managers how much an empirical loss can differ from its expected counterpart. 

In this work, we hope to bring the main message that besides the selection of a robust loss, the selection of a good proxy also matters for effective comparison of predictors, especially when the sample size $T$ is not large enough. 
For a single time point, we show that the probability of making false comparison could be very high. So the natural question is that by averaging the performance comparison over $T$ time points, are we able to get a faithful comparison of two predictors with high probability, so that the empirical loss ranking does reflect the population loss ranking? The answer is that we need robust proxies in order to have this kind of guarantee. 

We propose three robust proxies and compare them. The first choice uses the clipped squared return at the single current time $t$ as the proxy. This may be the simplest practical choice of a robust proxy. However, it cannot achieve the desired convergence in terms of empirical risk comparison due to large variance of only using a single time point. 
The second option mimics the EWMA proxy, so now we clip and average at multiple time points close to $t$. To find out the proper clipping, we first run EWMA tuning-free Huber loss minimization on local data for each time $t$. This will give a truncation level adaptive to the unknown volatility. Then the clipping bound will be rescaled to reflect the total sample size. According to literature on Huber minimization \cite{Cat12, FLW16, sun2020adaptive}, the truncation level needs to scale with the square root of the sample size to balance the bias and variance optimally. Therefore, it is natural to rescale the clipping bound by square root of the ratio of the total sample size $T$ over the local effective sample size. 
%The truncation levels will be rescaled and adapt to the changing volatility levels, thanks to our tuning-free adaptive Huber approach. 
The third proxy exactly solves the EWMA Huber minimization, again with the rescaled truncation. Compared to the first and second proxies, this gives further improvement on the deviation bound of the proxy, depending on the central kurtosis rather than the absolute kurtosis. We will illustrate the above claims in more detail in later sections. 

The Huber loss minimization has been proposed by \cite{huber1964robust} under Huber's $\epsilon$-contamination model and its asymptotic properties have been studied in \cite{huber1973robust}. At that time, the truncation level was set as fixed according to the $95\%$ asymptotic efficiency rule and ``robustness'' means achieving minimax optimality under the $\epsilon$-contamination model \citep{chen2018robust}. But recently, Huber's M-estimator has been revisited in the regression setting under the assumption of general heavy-tailed distributions \citep{Cat12, FLW16}. Here ``robustness'' slightly changes its meaning to achieving sub-Gaussian non-asymptotic deviation bound under the heavy-tailed data assumption. In this setting, the truncation level grows with sample size and the resultant M-estimator is still asymptotically unbiased even when data distribution is asymmetric. Huber's estimator fits the goal of robust volatility prediction and robust proxy construction very well, as squared returns indeed have asymmetric distributions. Since \cite{Cat12}, new literature to reveal deeper understanding on Huber's M-estimator sprung up. For example, \cite{sun2020adaptive} proved the necessity of finite fourth moment for volatility estimation if we hope to achieve a sub-Gaussian type of deviation bound; \cite{wang2020new} proposed the tuning-free Huber procedure; \cite{chen2018robust, minsker2018sub} extended the Huber methodology to robust covariance matrix estimation. 

Robustness issue is indeed an important concern for real data volatility forecasting. It has been widely observed that financial returns have fat tails. When it comes to the crypto markets e.g. Bitcoin product (BTC), the issue gets more serious, as crypto traders frequently experience huge jumps in the BTC price. For example, BTC plummeted more than 20\% in a single day in March 2020. The lack of government regulation probably leaves the market far from efficient. 
This posts a stronger need for robust methodology to estimate and forecast volatility for crypto markets. Some recent works include \cite{catania2018predicting, trucios2019forecasting, charles2019volatility}. 

With the BTC returns, we will compare the non-robust EWMA predictor with the robust Huber predictor, with different decays, and evaluate their performance using the non-robust forward EWMA proxy and the robust forward Huber proxy. Both the predictors and proxies will be rolled forward and compared at the end of each day. We apply two robust losses, MSE and QL, to evaluate their performance. Interestingly, we will see that when sample size $T$ is large, our proposed robust proxy will be very close to forward EWMA proxy, and both will lead to sensible and similar comparison. However, when $T$ is small, non-robust proxy could lead to higher probability of making wrong conclusions, whereas the robust proxy, which automatically adapts to the total sample size and the time-varying volatilities, can still work as expected. This matches with our theoretical findings and provides new insights about applying robust proxies for practical risk evaluation. 

The rest of the paper is organized as follows. In Section \ref{sec2}, we first review the definition of robust loss by \cite{patton2011volatility} and explain our analyzing strategy for high probability bound of the empirical loss. We bridge the empirical loss and the unconditional expected loss, by the conditional expected loss conditioning on proxies.
In Section \ref{sec3}, we propose three robust proxies and prove that they can all achieve the correct ranking with high probability, if measured by  the conditional expected loss. 
However, the proxy based on Huber loss minimization will have the smallest probability of making false comparison, if measured by the empirical loss.
In Section \ref{sec4}, we will discuss robust predictors and see why the above claim is true and why comparing robust predictors with non-robust predictors can be a valid thing to do. Simulation studies as well as an interesting case study on BTC volatility forecasting are presented in Section \ref{sec5}. We finally conclude the paper with some discussions in Section \ref{sec6}. All the proofs are relegated to the appendix.

\section{Evaluation of volatility forecast}\label{sec2}
In this section, we first review the key conclusions of \cite{patton2011volatility} on robust loss functions for volatility forecast comparison. 
%We then propose robust forward volatility proxies and elucidate why they yield more reliable gauges of volatility prediction error from an empirical deviation perspective.
We then use examples to see why we also care about the randomness from proxy deviation beyond picking a robust loss.

\subsection{Robust loss functions}\label{sec2.1}
Suppose we have a time series of returns $(X_i)_{-\infty < i < \infty}$ of a financial instrument. Let $\mathcal F_{t-1}$ denote the $\sigma$-algebra generated from $(X_i)_{i \le t - 1}$. Consider a volatility predictor $h_t$, computed at time $t$ based on $\mathcal F_{t-1}$, that targets $\sigma ^ 2_t := \var(X_t)$. We use a loss function $L(\sigma ^ 2_t, h_t)$ to gauge the prediction error of $h_t$. In practice, we never observe $\sigma^2_t$; therefore, in order to evaluate the loss function $L(\sigma^2_t, h_t)$, we have to substitute $\sigma_t^2$ therein with a proxy $\hat\sigma_t^2$, which is computed based on $\mathcal{G}_t$, the $\sigma$-algebra generated from the future returns $(X_t,\dots,X_{\infty})$. 
% Assume we have a loss function $L(\sigma_t^2, h_t)$ where $h_t$ is , which is the $\sigma$-algebra generated from say returns of an instrument $(X_{-\infty}, \dots, X_{t-1})$, and . 

Following \cite{patton2011volatility}, to achieve reliable evaluation of volatility forecasts, we wish to have the loss function $L$ satisfy the following three desirable properties: 

% defines $L$ as a robust loss function if $L$ yields consistent ranking of any two volatility forecasts when $\sigma^2_t$ is replaced with any unbiased proxy $\hat\sigma_t ^ 2$. 
% Formally, we say $L$ is robust if for any proxy $\hat\sigma_t ^ 2$ satisfying $\EE \hat\sigma_t ^ 2 = \sigma_t ^ 2$, $\EE L(\sigma_t ^ 2, h_t) < \EE L(\hat\sigma_t ^ 2, h_t)$ implies that $\EE L(\sigma_t ^ 2, h_t) < \EE L(\hat\sigma_t ^ 2, h_t)$.

\begin{itemize}
\item[(a)] \emph{Mean-pursuit}: $h_t^* := \argmin_{h \in \mathcal H} \EE[L(\hat\sigma_t^2, h) | \mathcal F_{t-1}] = \EE[\hat\sigma_t^2 | \mathcal F_{t-1}]$. This says that the optimal predictor is exactly the conditional expectation of the proxy.

\item[(b)] \emph{Proxy-robust}: Given any two predictors $h_{1t}$ and $h_{2t}$ and any unbiased proxy $\hat\sigma_t^2$, i.e., $E[\hat\sigma_t^2 | \mathcal F_{t-1}] = \sigma_t^2$,  $\EE\{L(\sigma_t^2, h_{1t})\} \le \EE\{L(\sigma_t^2, h_{2t})\} \iff  \EE\{L(\hat\sigma_t^2, h_{1t})\} \allowbreak \le \EE\{L(\hat\sigma_t^2, h_{2t})\}$. This means that the forecast ranking is robust to the choice of the proxy.

\item[(c)] \emph{Homogeneous}: $L$ is a homogeneous loss function of order $k$, i.e., $L(a \sigma^2, a h) = a^k L(\sigma^2, h)$ for any $a > 0$. This ensures that the ranking of two predictors is invariant to the re-scaling of data.
\end{itemize}

Define the mean squared error (MSE) loss and quasi-likelihood (QL) loss as 
\beq
    \label{eq:mse_ql}
    \mse(\sigma^2, h)  = (\sigma^2 - h)^2 \text{~and~} \ql(\sigma^2, h)  = \frac{\sigma^2} {h} - \log\bigg( \frac{\sigma^2} {h} \bigg) - 1
\eeq
respectively. Here the QL loss can be viewed, up to an affine transformation, as the negative log-likelihood function of $X$ that follows $\cN(0, h)$ when we observe that $X ^ 2 = \sigma ^ 2$. Besides, QL is always positive and the Taylor expansion gives that $\ql(\sigma^2, h) \approx  (\sigma^2 / h -1)^2 / 2$ when $\sigma ^ 2 / h$ is around $1$. \cite{patton2011volatility} shows that among many commonly used loss functions, MSE and QL are the only two that satisfy all the three properties above. Specifically, Proposition 1 of \cite{patton2011volatility} says that given that $L$ satisfies property (a) and some regularity conditions, $L$ further satisfies property (b) if and only if $L$ takes the form: 
\beq
    \label{eq:loss_form}
    L(\sigma^2, h) = \tilde C(h) + B(\sigma^2) + C(h) (\sigma^2 - h),
\eeq
where $C(h)$ is the derivative function of $\tilde C(h)$ and is monotonically decreasing.  Proposition 2 in \cite{patton2011volatility} establishes that MSE is the only proxy-robust loss that depends on $\sigma^2 - h$ and that QL is the only proxy-robust loss that depends on $\sigma^2 / h$. Finally, Proposition 4 in \cite{patton2011volatility} gives the entire family of proxy-robust and homogeneous loss functions, which include QL and MSE (MSE and QL are homogeneous of order 2 and 0 respectively). Given such nice properties of MSE and QL, we mainly use MSE and QL to evaluate and compare volatility forecasts throughout this work. 
% Then we may consider extending to other losses with homogeneous order other than 2 and 0.

\subsection{The empirical deviation perspective}

Besides selecting a robust loss as \cite{patton2011volatility} suggested, one has to also nail down the proxy selection for prediction loss computation. \cite{patton2011volatility}'s framework did not separate the randomness from the predictors and the proxies, and the proxy-robust property (b) compares two predictors in long-term unconditional expectation, which averages both randomnesses. However, it is not clear from \cite{patton2011volatility} that for a given selected proxy, what is the probability that we end up a wrong comparison of two predictors. How does the random deviation of a proxy affect the comparison? Can some proxies outperform others in terms of less probability to make mistakes in finite sample? 

In practice, one has to use empirical risk to approximate the expected risk to evaluate volatility forecasts. This implies one important issue that property (b) neglects: Property (b) concerns only the expected risk and ignores the deviation of the empirical risk from its expectation. Such empirical deviation is further exacerbated by replacing the true volatility with its proxies, jeopardizing accurate evaluation of volatility forecasts. 
% especially when we use volatility proxies to replace the true variances, the unconditional mean is for both $\mathcal F_{t-1}$ and $\mathcal G_t$, we may not want to ignore the variance from proxies, especially when we do not control the variance of the proxy well enough. 
Our strategy of analysis is as follows: we first link the empirical risk to the conditional risk (conditioning on the selected proxy), claiming that they are close with high probability (see formal arguments in Section \ref{sec4}), and then study the relationship of comparing the unconditional risk and conditional risk.

Specifically, we are interested in comparing the accuracy of two series of volatility forecasts $\{h_{1t}\}_{t \in [T]}$ and $\{h_{2t}\}_{t \in [T]}$. For notational convenience, we drop the subscript ``$t \in [T]$'' when we refer to a time series unless specified otherwise. Define $\eta_t = \EE\{L(\sigma_t^2, h_{2t}) - L(\sigma_t^2, h_{1t})\} > 0$.  Without loss of generality, suppose that $\{h_{1t}\}$ outperforms $\{h_{2t}\}$ in terms of expected loss, i.e., 
\beq
    \label{eq:expected_gap}
    \frac{1}{T} \sum_{t=1}^T \bigg[\EE L(\sigma_t^2, h_{2t}) - \EE L(\sigma_t^2, h_{1t})\bigg] = \frac{1}{T} \sum_{t=1}^T \eta_t > 0.
\eeq
The empirical loss comparison can be decomposed into the conditional loss comparison and the difference between empirical loss and conditional loss.
\begin{equation*}
\begin{aligned}
\frac{1}{T} \sum_{t=1}^T  \bigg(L(\hat\sigma_t^2, h_{2t}) & - L(\hat\sigma_t^2, h_{1t})\bigg) = \frac{1}{T} \sum_{t=1}^T \bigg[\bigg(\EE\{L(\hat\sigma_t^2, h_{2t})|\mathcal G_t\} - \EE\{L(\hat\sigma_t^2, h_{1t})|\mathcal G_t\} \bigg) \\
& + \bigg(L(\hat\sigma_t^2, h_{2t}) - \EE\{L(\hat\sigma_t^2, h_{2t})|\mathcal G_t\} \bigg) - \bigg(L(\hat\sigma_t^2, h_{1t}) - \EE\{L(\hat\sigma_t^2, h_{1t})|\mathcal G_t\} \bigg)\bigg] \,.
\end{aligned}
\end{equation*}
Therefore, we study the following two probabilities for any $\varepsilon > 0$:
\beq
\text{\rnum{1}} := \PP  \bigg[\frac{1}{T} \sum_{t=1}^T \bigg( \EE\{L(\hat\sigma_t^2, h_{2t}) | \mathcal G_t\} - \EE\{L(\hat\sigma_t^2, h_{1t}) | \mathcal G_t\} \bigg) < \frac{1}{T} \sum_{t=1}^T \eta_t - \varepsilon\bigg]\,,
\eeq
\beq
    \label{eq:II}
    \text{\rnum{2}} := \PP\bigg[\bigg|\frac{1}{T} \sum_{t=1}^T \bigg( L(\hat\sigma_t^2, h_{t}) - \EE\{L(\hat\sigma_t^2, h_{t})|\mathcal G_t\} \bigg) \bigg| > \varepsilon \bigg] \,.
\eeq
We aim to select stable proxies to make \rnum{1} small, so that the probability of obtaining false rank of the empirical risk is small. Meanwhile, with a selected proxy, we hope \rnum{2} can be well controlled for the predictors we care to compare.
% The reason we care about the above concentration for even non-robust predictors is that we hope to bridge the average loss and average expected loss (expectation is with respect to predictors but not proxies) as in term \rnum{2} for non-robust predictors as well so that the comparison between robust and non-robust predictors is also valid. Again we hope to point out that when we compare two predictors using average loss $\frac{1}{T} \sum_{t=1}^T L(\hat\sigma_t^2, h_{1t})$ and $\frac{1}{T} \sum_{t=1}^T L(\hat\sigma_t^2, h_{2t})$, there are randomness from both predictors and proxies. This paper provides an analysis to separate those two randomness: we first compare  $\frac{1}{T} \sum_{t=1}^T L(\hat\sigma_t^2, h_{it})$ with $\frac{1}{T} \sum_{t=1}^T \EE[L(\hat\sigma_t^2, h_{it})|\mathcal G_t]$ and bound the exception probability term in \rnum{2}, which only requires the concentration around expected loss; then we compare  $\frac{1}{T} \sum_{t=1}^T \EE[L(\hat\sigma_t^2, h_{1t})|\mathcal G_t]$ and $\frac{1}{T} \sum_{t=1}^T \EE[L(\hat\sigma_t^2, h_{2t})|\mathcal G_t]$, which requires good robust proxies in order to bound \rnum{1} and maintain the ranking of the two predictors with large probability when we use the true variance $\sigma_t^2$. 
Note that only randomness from proxy matters in \rnum{1}. So we can focus on proxy design by studying this quantity. Then we would like to make sure the difference between the empirical risk and conditional risk are indeed small via studying \rnum{2}. The probability in \rnum{2} is with respect to both the proxy and the predictor. By following this analyzing strategy, we separate the randomness from predictor and proxy and eventually give results on empirical deviation rather than in expectation.

\subsection{False comparison due to proxy randomness}

To illustrate this issue, we first focus on a single time point $t$. We compare two volatility forecasts $h_{1t}$ and $h_{2t}$ satisfying that $\EE\{L(\sigma_t^2, h_{1t})\} < \EE\{L(\sigma_t^2, h_{2t})\}$. We are interested in the probability of having a reverse rank of forecast precision between $h_{1t}$ and $h_{2t}$, conditioning on the selected proxy, i.e, $\PP_{\mathcal G_t} \big[\EE\{L(\hat\sigma_t^2, h_{1t}) | \mathcal G_t\} > \EE\{L(\hat\sigma_t^2, h_{2t}) | \mathcal G_t\}\big]$.
Note that this probability is with respect to the randomness of the proxy $\hat\sigma_t^2$; in the sequel, we show that it may not be small for a general proxy. 
But if we can select a good proxy to control this probability well, we can ensure a correct comparison with high probability.

% Given the heavy tail of the , there is no guarantee that any proxy can make the above probability small. \zzw{Which part is }
% A natural question arises: How could we mitigate the potential inconsistency between the ranks based on the empirical risk and the expected risk? 
% by employing robust loss functions as propsoed in \cite{patton2011volatility} as well as robust volatiliy proxies? 

Now consider MSE and QL as the loss functions, so that we can derive explicitly the condition for the empirical risk comparison to be consistent with the expected risk comparison.
% \subsection{The case of MSE and QLIKE}
% Let us illustrate the issue more clearly with the examples of robust loss MSE and QLIKE. 
% For simplicity, we first focus on only one time point $t$ and ignore the time dependency and 
For simplicity, assume that $\mathcal F_{t-1}$ and $\mathcal G_t$ are independent. Recall that $\eta_t = \EE\{L(\sigma_t^2, h_{2t}) - L(\sigma_t^2, h_{1t})\} > 0$. We wish to calculate $\PP_{\mathcal G_t} [\EE\{L(\hat\sigma_t^2, h_{2t}) | \mathcal G_t\} - \EE\{L(\hat\sigma_t^2, h_{1t}) | \mathcal G_t\} < \eta_t - \varepsilon_t]$ for some $\varepsilon_t >  |\eta_t|$, i.e., the probability of having the forecast rank in conditional expectation be opposite of the rank in unconditional expectation. When $L$ is chosen to be MSE, we have
\beq
    \label{eq:eta_t_mse}
    \eta_t = \EE L(\sigma_t^2, h_{2t}) - \EE L(\sigma_t^2, h_{1t}) = \EE(h_{2t}^2 - h_{1t}^2) + 2\sigma_t^2 \EE(h_{1t} - h_{2t})
\eeq
and
\[
\EE\{L(\hat\sigma_t^2, h_{2t}) | \mathcal G_t\} - \EE\{L(\hat\sigma_t^2, h_{1t}) | \mathcal G_t\} = \EE(h_{2t}^2 - h_{1t}^2) + 2 \hat\sigma_t^2 \EE(h_{1t} - h_{2t}). 
\]
Therefore, 
\begin{equation*}
\begin{aligned}
\PP_{\mathcal G_t} & [\EE\{L(\hat\sigma_t^2, h_{2t}) | \mathcal G_t\} - \EE\{L(\hat\sigma_t^2, h_{1t}) | \mathcal G_t\} < \eta_t - \varepsilon_t] \\
& = \PP_{\mathcal G_t} \{2 (\hat\sigma_t^2 - \sigma_t^2) \EE(h_{1t} - h_{2t}) < -\varepsilon_t\} \\
& = \PP_{\mathcal G_t} [(\hat\sigma_t^2 / \sigma_t^2 - 1) \{\eta_t - \EE(h_{2t}^2 - h_{1t}^2)\} < -\varepsilon_t]. 
\end{aligned}
\end{equation*}
For illustration purposes, consider a deterministic scenario where $h_{1t} = \sigma_t^2$ is the oracle predictor, and where $h_{2t} = \sigma_t ^ 2 + \sqrt \eta_t$ (so that \eqref{eq:eta_t_mse} holds). 
% One can verify that If $b > 0$ (upward biased), it is straightforward to see $b = \eta_t + 2\sigma_t^2 \sqrt{\eta_t}$ to achieve $\eta_t$ advantage in $E[L(\sigma_t^2, h_{2t})] - E[L(\sigma_t^2, h_{1t})] = \eta_t$. 
Then
\begin{equation*}
\begin{aligned}
    \PP_{\mathcal G_t} & \bigl[\EE\{L(\hat\sigma_t^2, h_{2t}) | \mathcal G_t\} - \EE\{L(\hat\sigma_t^2, h_{1t}) | \mathcal G_t\} < \eta_t - \varepsilon_t\bigr] = \PP_{\mathcal G_t} \bigg( \frac{\hat\sigma_t^2}{\sigma_t^2}  - 1 > \frac{\varepsilon_t}{2\sigma_t^2 \sqrt{\eta_t}} \bigg). 
\end{aligned}
\end{equation*}
Similarly, if $h_{2t} = \sigma_t ^ 2 - \sqrt \eta_t$, we have 
% $b = \eta_t - 2\sigma_t^2 \sqrt{\eta_t}$, and
\begin{equation*}
\begin{aligned}
\PP_{\mathcal G_t} & \bigl[\EE\{L(\hat\sigma_t^2, h_{2t}) | \mathcal G_t\} - \EE\{L(\hat\sigma_t^2, h_{1t}) | \mathcal G_t\} < \eta_t - \varepsilon_t\bigr] = \PP_{\mathcal G_t} \bigg( \frac{\hat\sigma_t^2}{\sigma_t^2} - 1 < - \frac{\varepsilon_t}{2\sigma_t^2 \sqrt{\eta_t}}\bigg).
\end{aligned}
\end{equation*}
We can see from the two equations above that a large deviation of $\hat\sigma_t ^ 2$ from $\sigma_t ^ 2$ gives rise to inconsistency between forecast comparisons based on empirical risk and expected risk. When we choose $L$ to be QL, we have that 
% With robust proxy $\hat\sigma_t^2$, we expect to bound the above two probabilities well. Otherwise, we may have a non-negligible probability to reach a wrong conclusion.
% \noindent {\bf \underline{QLIKE}}
$$
    \eta_t = \EE L(\sigma_t^2, h_{2t}) - \EE L(\sigma_t^2, h_{1t}) = \EE(\log h_{2t} - \log h_{1t}) + \sigma_t^2 \EE\bigg(\frac{1}{h_{2t}} - \frac{1}{h_{1t}}\bigg)
$$
and that
\begin{equation*}
\begin{aligned}
    \PP_{\mathcal G_t} & \bigl[\EE\{L(\hat\sigma_t^2, h_{2t}) | \mathcal G_t\} - \EE\{L(\hat\sigma_t^2, h_{1t}) | \mathcal G_t\} < \eta_t - \varepsilon_t\bigr] \\
    & = \PP_{\mathcal G_t} \bigg\{(\hat\sigma_t^2 - \sigma_t^2) \EE\bigg(\frac{1}{h_{2t}} - \frac{1}{h_{1t}}\bigg) < -\varepsilon_t\bigg\} \\
    & = \PP_{\mathcal G_t} \big[(\hat\sigma_t^2 / \sigma_t^2 - 1)\{\eta_t - \EE(\log h_{2t} - \log h_{1t})\} < -\varepsilon_t\big].
\end{aligned}
\end{equation*}
Similarly, we consider a deterministic setup where $h_{1t} = \sigma_t^2$, and where $h_{2t} = m \sigma_t^2$ with a misspecified scale. To ensure that $\EE L(\sigma_t^2, h_{2t}) - \EE L(\sigma_t^2, h_{1t}) = \eta_t$, we have $\eta_t = 1/m - \log(1/m) - 1 \approx (1/m -1)^2/2$ when $1 / m \approx 1$. In this case, we deduce that
\begin{equation*}
\begin{aligned}
    \PP_{\mathcal G_t} & \{\EE(L(\hat\sigma_t^2, h_{2t}) | \mathcal G_t) - \EE(L(\hat\sigma_t^2, h_{1t}) | \mathcal G_t) < \eta_t - \varepsilon_t\} = \PP_{\mathcal G_t} \{ (\hat\sigma_t^2 / \sigma_t^2 - 1)(\eta_t - \log m) < -\varepsilon_t\}  \\
    & = \PP_{\mathcal G_t} \{(\hat\sigma_t^2 / \sigma_t^2 - 1)(1/m - 1) < -\varepsilon_t\} \approx \begin{cases}
      \PP_{\mathcal G_t} \bigl\{\hat\sigma_t^2 / \sigma_t^2 - 1 > \frac{\varepsilon_t}{\sqrt{2\eta_t}}\bigr\}  & \text{if $m > 1$,}\\
      \PP_{\mathcal G_t} \bigl\{\hat\sigma_t^2 / \sigma_t^2 - 1 < - \frac{\varepsilon_t}{\sqrt{2\eta_t}}\bigr\}  & \text{if $m \le 1$.}
    \end{cases}   
\end{aligned}
\end{equation*}
Similarly, we can see that the volatility forecast rank will be flipped once the deviation of $\hat\sigma_t ^2$ from $\sigma_t ^ 2$ is large.

% \subsection{How serious is the issue?}
% In practice, we typically compare two close predictors, meaning that $\eta_t$ is typically small. Consider $\varepsilon_t = 2\eta_t$ so that with the proxy variance we actually flip sign of the achieved loss difference. In QLIKE loss, the probability becomes $\PP_{\mathcal G_t} ( \hat\sigma_t^2 / \sigma_t^2 - 1 > \sqrt{2\eta_t} )$. With the proxy $\hat\sigma_t^2 = X_t^2$, when $X_t \sim N(0, \sigma_t^2)$ or $X_t \sim t_{4}(0, \sigma_t^2)$, we can calculate this probability to be as large as $31.7\%$ or $37.4\%$ if we let $\eta_t \to 0$. 
% The issue can be more serious when using MSE. For MSE, we will reach the above probability of false conclusion if $\eta_t / \sigma_t^4 \to 0$. This can be easier to see when we have large true $\sigma_t^2$ say during financial crisis regime. 

Note again that in the derivation above, the probability of reversing the expected forecast rank is evaluated at a single time point $t$, which is far from enough to yield reliable comparison between volatility predictors. The common practice is to compute the empirical average loss of the predictors over time for their performance evaluation. Two natural questions arise: Does the empirical average completely resolve the instability of forecast evaluation due to the deviation of volatility proxies? If not, how should we robustify our volatility proxies to mitigate their empirical deviation?

\section{Robust volatility proxies} \label{sec3}

\subsection{Problem setup} \label{sec3.1}

Our goal in this section is to construct robust volatility proxies $\{\hat\sigma^2_t\}$ to ensure that $\{h_{1t}\}$ maintains empirical superiority with high probability, or more precisely, that $\PP_{\mathcal G_t}\big[\frac{1}{T} \sum_{t=1}^T \allowbreak\EE\{L(\hat\sigma_t^2, h_{2t})|\mathcal G_t\} - \EE\{L(\hat\sigma_t^2, h_{1t})|\mathcal G_t\} < 0\big]$ is small, given $\frac{1}{T} \sum_{t=1}^T \eta_t > 0$.
We first present our assumption on the data generation process. 
\begin{assumption}\label{assump1}
Given the true volatility series $\{\sigma_t\}$, instrument returns $\{X_t\}$ are independent with $\EE X_t = 0$ and $\mathrm{Var}(X_t) = \sigma_t^2$. The central and absolute fourth moments of $X_t$, denoted by $\kappa_t = \EE\{(X_t^2 - \sigma_t^2)^2\}$ and $\tilde\kappa_t = \EE X_t^4$, are both finite. 
\end{assumption}
Now we introduce some quantities that frequently appear in the sequel. At time $t$, define the smoothness parameters
\beq
    \begin{aligned}
        & \delta_{0,t} := \sum_{s=t}^{t+m} w_{s, t}\sigma_s^2 - \sigma_{t}^2,  ~~~\delta_{1,t} := \sum_{s=t}^{t+m} w_{s, t}^2 |\sigma_s^2 - \sigma_{t}^2|^2, \\
        & \Delta_{0,t} := \sum_{s=t - m}^{t-1} \nu_{s, t} \sigma_s^2 - \sigma_{t}^2 ~\text{ and }~ \Delta_{1,t} = \sum_{s=t-m}^{t-1} \nu_{s, t}^2 |\sigma_s^2 - \sigma_{t}^2|^2, 
    \end{aligned}
\eeq
% $$
% \delta_{0,t} := \sum_{s=t}^{t+m} w_{s,t} \sigma_s^2 - \sigma_{t}^2 = o\bigg\{\bigg(\sum_{s=t}^{t+m} w_{s,t}^2\bigg) ^ {1 / 2}\bigg\}\;\; \text{and} \;\; \delta_{1,t} := \sum_{s=t}^{t+m} w_{s,t}^2 |\sigma_s^2 - \sigma_{t}^2|^2 = o\bigg(\sum_{s=t}^{t+m} w_{s,t}^2\bigg),
% $$
where $w_{s, t} = \lambda ^ {s - t} / \sum_{j = t}^{t + m} \lambda ^ {j - t}$ is the forward exponential-decay weight at time $s$ from time $t$ with rate $\lambda$, and where $\nu_{s, t} = \lambda^{t-1-s} / \sum_{s=t-m}^{t-1} \lambda^{t-1-s}$ is the backward exponential-decay weight with rate $\lambda$. These smoothness parameters characterize how fast the distribution of volatility varies as time evolves, and our theory explicitly derives their impact. As we shall see, our robust volatility proxies yield desirable statistical performance as long as these smoothness parameters are small, meaning that the variation of the volatility distribution is slow. Besides, define the forward and backward effective sample sizes as 
\beq
    \label{eq:n_eff}
    n^\dagger_{\eff} := 1 / \sum_{s=t}^{t+m} w_{s, t}^2~~\text{and}~~n^\ddagger_{\eff} := 1 / \sum_{s=t - m}^{t - 1} \nu_{s, t}^2
\eeq
respectively, and define the forward and backward exponential-weighted moving average (EWMA) of the central fourth moment as 
\beq
    \label{eq:kappa}
    \kappa ^ {\dagger}_t = \sum_{s=t}^{t+m} w_{s, t}^2 \kappa_{s} /\allowbreak \sum_{s=t}^{t+m} w_{s, t}^2~~\text{and}~~\kappa ^ {\ddagger}_t = \sum_{s=t}^{t+m} \nu_{s, t}^2 \kappa_{s} /\allowbreak \sum_{s=t}^{t+m} \nu_{s, t}^2
\eeq
respectively. Similarly, we have $\tilde\kappa ^ {\dagger}_t$ and $\tilde \kappa ^ {\ddagger}_t$ as the forward and backward EWMA of the absolute fourth moment. 

Consider a mean-pursuit and proxy-robust loss function that takes the form \eqref{eq:loss_form}:
$$
    L(\sigma^2, h) = \tilde C(h) + B(\sigma^2) + C(h) (\sigma^2 - h) =: f(h) + B(\sigma^2) + C(h) \sigma^2, 
$$
where we write $f(h) = \int_{a}^{h} C(x)dx - C(h)h$ for any constant $a$. When $C(h)= -2h$ and $f(h) = h^2$ ($a=0$), $L$ is MSE. When $C(h) = 1/h$ and $f(h) = \log h$ ($a=e^{-1}$), $L$ becomes QL. Under Assumption \ref{assump1}, $h_{it}$ and $\hat\sigma_t^2$ are independent for $i = 1, 2$. Therefore, $\EE\{L(\hat\sigma_t^2, h_{it}) | \mathcal G_t\} = \EE[f(h_{it})] + B(\hat\sigma_t^2) + \EE[C(h_{it})] \hat\sigma_t^2$. Given \eqref{eq:expected_gap}, we wish to show that $\{h_{2t}\}$ outperforms $\{h_{1t}\}$ in conditional risk with high probability, i.e. \rnum{1} is small. Recall that  
% Therefore, we need to make sure the difference of $\EE[L(\hat\sigma_t^2, h_{it}) | \mathcal G_t]$ between the two predictors are not far from the population loss with true variance, i.e. the following probability must be small:
\begin{equation*}
\begin{aligned}
    \text{\rnum{1}} = & \PP  \bigg[\frac{1}{T} \sum_{t=1}^T \EE\{L(\hat\sigma_t^2, h_{2t}) | \mathcal G_t\} - \EE\{L(\hat\sigma_t^2, h_{1t}) | \mathcal G_t\} < \frac{1}{T} \sum_{t=1}^T \eta_t - \varepsilon\bigg] \\
    % = \PP&\bigg[\frac{1}{T} \sum_{t=1}^T \big\{\big(\EE f(h_{2t}) + B(\hat\sigma_t^2) + \EE[C(h_{2t})] \hat\sigma_t^2\big) - \big(\EE[f(h_{1t})] + B(\hat\sigma_t^2) + \EE[C(h_{1t})] \hat\sigma_t^2\big)\big\} <  \frac{1}{T} \sum_{t=1}^T \eta_t - \varepsilon \bigg] \\
    = & \PP\bigg[\frac{1}{T} \sum_{t=1}^T \EE\{f(h_{2t}) - f(h_{1t})\} + \EE\{C(h_{2t}) - C(h_{1t})\} \hat\sigma_t^2 <  \frac{1}{T} \sum_{t=1}^T \eta_t - \varepsilon \bigg] \\
    = & \PP\bigg[\frac{1}{T} \sum_{t=1}^T (\hat\sigma_t^2/\sigma_t^2 - 1)(\eta_t - A_t) <  - \varepsilon \bigg], 
\end{aligned}
\end{equation*}
where $\varepsilon > 0$ is a deviation parameter that may exceed $T ^ {-1}\sum_{t \in [T]} \eta_t$, and the last equation is due to the fact that $\eta_t = \EE[f(h_{2t}) - f(h_{1t})] + \EE[C(h_{2t}) - C(h_{1t})] \sigma_t^2$. 
% We also have . Define $A_t = \EE[f(h_{2t}) - f(h_{1t})], B_t = \EE[C(h_{2t}) - C(h_{1t})] $. The above equals
% Therefore, to bound $\PP[\frac{1}{T} \sum_{t=1}^T (\EE[L(\hat\sigma_t^2, h_{2t}) | \mathcal G_t] - \EE[L(\hat\sigma_t^2, h_{1t}) | \mathcal G_t]) <  \frac{1}{T} \sum_{t=1}^T \eta_t - \varepsilon]$, we need to control \rnum{1}.

\subsection{Exponentially weighted Huber estimator}

We first review the tuning-free adaptive Huber estimator proposed in \cite{wang2020new}. Define the Huber loss function $\ell_{\tau}(x)$ with robustification parameter $\tau$ as
\[
\ell_\tau(x) := \begin{cases}
        \tau x - \frac{\tau ^ 2}{2}, \quad\text{if $x > \tau$;}
        \\
        \frac{x ^ 2}{2}, \quad\quad\quad\text{   if $|x| \le \tau$;}
        \\
        - \tau x - \frac{\tau ^ 2}{2}, \text{ if $x < -\tau$.}
        \end{cases}
\]
Suppose we have $n$ independent observations $\{Y_i\}_{i \in [n]}$ of $Y$ satisfying that $\EE Y = \mu$ and that $\var(Y) = \sigma^2$. The Huber mean estimator $\hat \mu$ is obtained by solving the following optimization problem:
\[
    \hat \mu := \argmin_{\theta \in \RR} \sum_{i = 1} ^ n \ell_{\tau}(Y_i - \theta). 
\]
\cite{fan2017estimation} show that when $\tau$ is of order $\sigma \sqrt n$, $\hat\mu$ achieves the optimal statistical rate with a sub-Gaussian deviation bound: 
\beq    
    \label{eq:subexp_bound}
    \PP(| \hat \mu - \mu| \lesssim \sigma \sqrt{z/n}) \ge 1-2e^{-z}, \forall z > 0. 
\eeq
In practice, $\sigma$ is unknown, and one therefore has to rely on cross validation (CV) to tune $\tau$, which incurs loss of sample efficiency.  \cite{wang2020new} propose a data-driven principle to estimate $\mu$ and the optimal $\tau$ jointly by iteratively solving the following two equations:
\beq
    \label{eq:tuning_free_huber}
    \left\{
    \begin{aligned}
        & \frac1n \sum_{i=1}^{n} \min(|Y_i - \theta|, \tau) \sgn(Y_i - \theta) = 0; \\
        & \frac1n \sum_{i=1}^{n} \frac{\min(|Y_i - \theta|^2, \tau^2)} {\tau^2} - \frac zn = 0, 
    \end{aligned}
    \right.
\eeq
where $z$ is the same deviation parameter as in \eqref{eq:subexp_bound}. 
Specifically, we start with $\hat\theta^{(0)} = \bar Y$ and solve the second equation for $\tau^{(1)}$. We then plug $\tau = \tau^{(1)}$ into the first equation to get $\hat\theta^{(1)}$. We repeat these two steps until the algorithm converges and use the final value of $\hat\theta$ as the estimator for $\mu$. \cite{wang2020new} proved that (i) if $\theta = \mu$ in the second equation above, then its solution gives $\tau = \sigma \sqrt{n/z}$ with probability approaching $1$; (ii) if we choose $\tau = \sigma \sqrt{n/z}$ in the first equation above, its solution satisfies \eqref{eq:subexp_bound}, even when $Y$ is asymmetrically distributed with heavy tails. Note that \cite{wang2020new} call the above procedure tuning-free, in the sense that the knowledge of $\sigma$ is not needed, but we still have the deviation parameter $z$ used to control the exception probability. The paper suggested to use $z = \log n$ in practice. 

In the context of volatility forecast, $\sigma_t$ always varies across time. The well known phenomenon of volatility clustering in the financial market implies that $\sigma_t$ typically changes slowly, so that we can borrow data around time $t$ to help with estimating $\sigma_t ^ 2$ with little bias. A common practice in quantitative finance is to exploit an exponential-weighted average of $\{X^2_s\}_{s = t} ^ {t + m}$ to estimate $\sigma_t ^ 2$, thereby discounting the importance of data that are distant from time $t$. To accommodate such exponential-decaying weights, we now propose a sample-weighted variant of the Huber estimator for volatility estimation as follows: 
\beq
    \label{eq:huber_weighted}
    \hat\mu := \argmin_{\theta \in \RR} \sum_{s = t} ^ {t + m} w_{s, t}\ell_{{\tau_t / w_{s, t}}}(X_s ^ 2 - \theta) =: \mathcal L_{\tau_t}(\theta; \{w_{s, t}\}_{s = t} ^ {t + m}), 
\eeq
% to construct robust volatility proxies. 
% One natural way of defining a robust proxy using a tuning-free  Huber's M-estiamtor is as follows. 
where $\{w_{s,t}\}_{s \in \{t, t + 1, \ldots, t + m\}}$ are the sample weights. Note that the robustification parameters for the observations can be different: intuitively, the higher the sample weight is, the lower $\tau$ should be, so that we can better guard against heavy-tailed deviation of important data points. More technical justification on such choice of robustification parameters is given after Theorem \ref{thm1}. 
% When we estimate $\sigma_t$, it is natural to put larger weights for time points closer to $t$ and then roll the estimation window forward to $t+1$. 
Correspondingly, to adaptively tune $\tau$, we iteratively solve the following two equations for $\tau_t$ and $\theta_t$ until convergence:
\beq
    \label{eq:tuning_free_huber_weighted}
    \left \{
    \begin{aligned}
        & \sum_{s=t}^{t+m} w_{s, t} \min\bigg(|X_s^2 - \theta_t|,  \frac{\tau_t}{w_{s, t}}\bigg) \sgn(X_s^2 - \sigma_t^2) = 0; \\
        & \sum_{s=t}^{t+m} w_{s, t}^2 \min\bigg(|X_s^2 - \theta_t|^2, \frac{\tau_t^2}{w_{s, t}^2}\bigg) / \tau_t^2 - z = 0.
    \end{aligned}
    \right.
\eeq
% We use the final iterate of $\theta_t$ as the volatility proxy. 
% for iid data in \cite{wang2020new}: (1) we do not assume data are i.i.d, in particular, we assume $\EE[X_s^2] = \sigma_s^2$ and leave the requirement for the smoothness of $\sigma_s^2$ explicit in the results; (2) we extend the procedure to general unequal weights, which are important in practice, since  Both extensions are less straightforward than it may look: note that the clipping bound should also depend on the weights. 
% {\color{blue} Do we need to give algorithm to solve this somewhere?}

Our first theorem shows that the solution to the first equation of \eqref{eq:tuning_free_huber_weighted} yields a sub-Gaussian estimator of $\sigma_t ^ 2$, provided that $\tau_t$ is well tuned and that the distribution evolution of the volatility series is sufficiently slow.
\begin{thm} \label{thm1}
Under Assumption \ref{assump1}, if $\tau_t =\sqrt{\frac{\kappa ^ {\dagger}_t}{n^\dagger_{\eff} z}}$, we have for $n^\dagger_{\eff} \ge 16z$ that 
\beq
    \label{eq:sigma_hat_bound}
    \PP\bigg(|\hat\sigma_t^2 - \sigma_t^2| \le 4 \sqrt{\kappa ^ {\dagger}_t z/n^\dagger_{\eff}}\bigg) \ge 1-2e^{-z+ |\delta_{0,t}| \sqrt{{n^\dagger_{\eff} z}/{\kappa ^ {\dagger}_t}} + 2\delta_{1,t}{n^\dagger_{\eff} z}/{\kappa ^ {\dagger}_t} }, 
\eeq
where $\hat\sigma_t^2$ equals $\theta_t$ that solves the first equation of \eqref{eq:tuning_free_huber_weighted}.
\end{thm}
\begin{rem}
    Given that $\{w_{s, t}\}_{s = t} ^ {t + m}$ are forward exponential-decay weights, we have that 
    \[
        n^\dagger_{\eff} = \frac{(1 + \lambda)(1 - \lambda ^ {m + 1})}{(1 - \lambda)(1 + \lambda ^ {m + 1})} = \frac{1 + \lambda}{1 + \lambda ^ {m + 1}}\sum_{i = 0} ^ m \lambda ^ i, 
    \]
    which converges to $(1 + \lambda) / (1 - \lambda)$ as $m \to \infty$, and which converges to $m + 1$ as $\lambda \to 1$. Therefore, $n^\dagger_{\eff} \to \infty$ requires both that $m \to \infty$ and that $\lambda \to 1$. As $n^\dagger_{\eff} \to \infty$, when $\delta_{0,t} = o(\sqrt{1/n^\dagger_{\eff}})$ and $\delta_{1,t} = o(1 / n^\dagger_{\eff})$, the exception probability is of order $e^{-z+o(z)}$, which converges to $0$ as $z \to \infty$. Therefore, if we choose $z = \log {n^\dagger_{\eff}}$, then $|\hat \sigma_t ^ 2 - \sigma_t ^ 2| = O_{\PP}\{(\log n^\dagger_{\eff} / n^\dagger_{\eff}) ^ {1 / 2}\}$ as $n^\dagger_{\eff} \to \infty$.
\end{rem}
\begin{rem}
    One crucial step of our proof is using Bernstein's inequality to control the derivative of the weighted Huber loss with respect to $\theta$, i.e., 
    \[
        \cL_\tau'(\theta; \{w_{s, t}\}) = \sum_{s=t}^{t+m} w_{s, t}\min(|X_s^2 - \theta|, \tau_t / w_{s, t}) \sgn(X_s^2 - \theta).
    \]
    Through setting $\tau_t / w_{s, t}$ as the robustification parameter for the data at time $s$, we can ensure that the corresponding summand in the derivative is bounded by $\tau_t$ in absolute value, which allows us to apply Bernstein's inequality. This justifies from the technical perspective our choices of the robustification parameters for different sample weights. 
\end{rem}
\begin{rem}
    Theorem \ref{thm1} is in fact not limited to exponential-decay weights $\{w_{s, t}\}_{s = t} ^ {t + m}$. Bound \eqref{eq:sigma_hat_bound} applies to any sample weight series. 
\end{rem}
 
% The effective sample size under Assumption \ref{assump2} is $n^\dagger_{\eff} = (1+\lambda)(1+\lambda+\dots+\lambda^m) / (1+\lambda^{m+1}) \to (1+\lambda)/(1-\lambda)$ as $m \to \infty$. To achieve a consistent proxy, we need the effective sample size to go to infinity, i.e. $\lambda \to 1$. This also indicates that either Assumption \ref{assump2} holds for a slow exponential average ($\lambda$ close to 1) and a longer horizon $m$ as $T \to \infty$, or if the smoothness of $\sigma_t^2$ does not change as $T$ grows, this requires the sampling frequency to go to infinity. 

% We comment on the truncation level here. Intuitively, for Huber loss from sample $s$ with a small weight $w_{s,t}$, the incurred the variance is also small so that we can apply a larger clipping bound of $\tau_t/w_{s,t}$ to incur a smaller bias. Technically, when we apply Bernstein inequality with equal weights to say $\sum_i X_i$, we require each data to be bounded by a constant $|X_i| \le M$ for the sub-exponential tail bound, which is why we clip using the same clipping. However, when the weights are not equal, to bound $\sum_i w_i X_i$, we require $|w_i X_i| \le M$, which leads to $|X_i| \le zM/w_i$. This helps to justify why the Huber clipping bound for each sample should also depend on the weight. Note that by applying this weight adaptive truncation, we do not need to make assumptions on the orders of the weights for Theorem \ref{thm1} to hold, although the commonly used EWMA weights are indeed on the same order if we requires $n^\dagger_{\eff} \to \infty$ or $\lambda \to 1$.

Our next theorem provides theoretical justification of the second equation of \eqref{eq:tuning_free_huber_weighted}. It basically says that the solution to that equation gives an appropriate value of $\tau_t$. 

\begin{thm} \label{thm2}
On top of Assumption \ref{assump1}, we further assume that $\delta_{1,t} n^\dagger_{\eff} \le c_1 \kappa ^ {\dagger}_t$ and $w_{s, t} \asymp 1 / m$ for all $t \le s \le t + m$. If $\theta_t = \sigma_t^2$ in the second equation of \eqref{eq:tuning_free_huber_weighted}, we have that as $n^\dagger_{\eff}, z \to \infty$ with $z = o(n^\dagger_{\eff})$, its solution $\tau_t \asymp \sqrt{\frac{\kappa ^ {\dagger}_t}{n^\dagger_{\eff}z}}$ with probability approaching $1$. 
\end{thm}
% The proof of Theorem \ref{thm2} unfortunately requires the weights to be on the same order. 
% So the adaptive tuning-free approach for now only works when all the data we give nonzero weights are of similar importance. 
\begin{rem}
    Define the half-life parameter $l := \log(1 / 2) / \log(\lambda)$. If we fix $m / l = C$ for a universal constant $C$, which is common practice in volatility forecast, then we can ensure that $\{w_{s, t}\}_{t \le s \le t + m}$ are all of order $1 / m$. 
\end{rem}
% Luckily, in practice, EWMA suffices for most of the applications and people seldom use weights of different orders, which essentially means we probably only care about using data with the high weight order for volatility forecasting.

\subsection{Average deviation of volatility proxies}\label{sec3.3}

We are now in position to evaluate \rnum{1}, which concerns average deviation of the volatility proxies over all the $T$ time points. To illustrate the advantage of the sample-weighted Huber volatility proxy as proposed in \eqref{eq:huber_weighted}, we first introduce and investigate two benchmark volatility proxies that are widely used in practice. Then we present our average deviation analysis of the sample-weighted Huber proxy.
% In the following, we first introduce these two bencmark volatility proxies, and then present theoretical results on the average deviation of all the three volatility proxies. 
% \subsubsection{Benchmark volatility proxies}

The first benchmark volatility proxy, which we denote by $(\hat \sigma_c)_t ^ 2$, is simply a clipped squared return: 
\beq
    \label{eq:bench_proxy_1}
    \left\{
    \begin{aligned}    
        & {(\hat\sigma_c)^2_t} = \min(X_t ^ 2, c_t) = \min(X_t^2, \tau_t\sqrt{n^\dagger_{\eff}T}); \\
        & \sum_{s=t}^{t+m} w_{s, t}^2 \min\bigg(X_s^4, \frac{\tau_t^2}{w_{s, t}^2}\bigg) /\tau_t^2 - z = 0,  
    \end{aligned}
    \right.
\eeq
where $c_t$ is the clipping threshold, and where $z$ is a similar deviation parameter as in \eqref{eq:tuning_free_huber_weighted}. 
Here $\tau_t$ is tuned similarly as in \eqref{eq:tuning_free_huber_weighted}, except that now the second equation of \eqref{eq:bench_proxy_1} does not depend on $\hat \sigma ^ 2_c$ and thus can be solved independently. 
% For convenience, write the clipping threshold $\tau_t \sqrt{n^\dagger_{\eff}T}$ as $c_t$.
% \noindent {\bf \underline{First option: single time clipping}}
% where $\hat\tau_t$ can be solved directly from the second equation. 
Following Theorem 2, we can deduce that $\tau_t \asymp \sqrt{\frac{{{\tilde\kappa_t} ^ \dagger}}{n^\dagger_{\eff}z}}$ and thus that $c_t \asymp \sqrt{\frac{{{\tilde\kappa ^ \dagger_t}} T}{z}}$. The main purpose of choosing such a rate of $c_t$ is to balance the bias and variance of the average of $(\hat\sigma_c)_t^ 2$ over $T$ time points. The following theorem develops a non-asymptotic bound for the average relative deviation of $\hat\sigma_c ^ 2$.

\begin{thm} \label{thm3}
    Under Assumption \ref{assump1}, if $c_t=\sqrt{{{\tilde\kappa^{\dagger}_t} T}/{z}}$ in \eqref{eq:bench_proxy_1}, for any bounded series $\{q_t\}_{t \in [T]}$ such that $\max_{t \in [T]}|q_t| \le Q$ and any $z > 0$, we have
    $$
        \PP\bigg(\frac1T \sum_{t=1}^T\{(\hat\sigma_c)_t^2 / \sigma_t^2 - 1\}q_t \le CQ \sqrt{z/T}\bigg) \ge 1-e^{-z}\,,
    $$
    where $C$ depends on $\max_{t \in [T]} \{(\tilde\kappa_t / \sigma_t^4) \vee ({\tilde{\kappa} ^ \dagger_t} / \sigma_t^4) \vee (\tilde\kappa_t /{{\tilde\kappa^{\dagger}_t}})\}$.
\end{thm}

% \noindent {\bf \underline{Second option: multiple time clipping}}
The second benchmark volatility proxy, which we denote by $(\hat\sigma_e)^2_t$, is defined as
\beq
    \label{eq:bench_proxy_2}
    \left \{
    \begin{aligned}
        & (\hat\sigma_e)_t^2 = \sum_{s=t}^{t+m} w_{s,t} \min\bigg(X_s^2, \frac{c_t}{w_{s,t}}\bigg) = \sum_{s=t}^{t+m} \min(w_{s,t} X_s^2, \tau_t \sqrt{T/n^\dagger_{\eff}}), \\
        & \sum_{s=t}^{t+m} w_{s,t}^2 \min\bigg(X_s^4, \frac{\tau_t^2}{w_{s,t}^2}\bigg) /\tau_t^2 - z = 0. 
    \end{aligned}
    \right.
\eeq
The second equation of \eqref{eq:bench_proxy_2} is the same as that of \eqref{eq:bench_proxy_1}. The only difference between $\hat\sigma ^ 2_e$ and $\hat\sigma ^ 2_c$ is that $\hat\sigma_e ^ 2$ exploits not only a single time point, but multiple data points in the near future to construct the volatility proxy. Accordingly, the clipping threshold is updated as $\tau_t\sqrt{T / n_\eff}$.
The following theorem characterizes the average relative deviation of $\hat\sigma ^ 2_e$. 
% Correspondingly, the clipping level is updated to $c_t = \hat\tau_t \sqrt{T/n^\dagger_{\eff}} \sim \sqrt{\frac{{\tilde\kappa^{\dagger}_t}T}{n^{\dagger 2}_{\eff} z}}$. 
% In practice, for the sake of simplicity, one may use the first option. 
% However, the first option is obviously going to incur more intrinsic variance in proxies in terms of comparison at a single point. If computation is not a concern, one may favor the second option. Here 

\begin{thm}\label{thm4}
Under Assumption \ref{assump1}, for any $(z, c) > 0$ satisfying that 
\begin{enumerate}
    \item $\max_{t \in [T]} \bigg\{ 2|\delta_{0,t}|/\sqrt{{\tilde\kappa^{\dagger}_t}} + 4\delta_{1,t} n^\dagger_{\eff}/({\tilde\kappa^{\dagger}_t}\sqrt{T})\bigg\} \le c$; 
    \item $n^\dagger_{\eff} \ge 2c^{-1}\{1 + (\log 2T)/z\} \sqrt{T}$; 
    \item $\sqrt{T} \ge 16 c z$, $T \ge 16z$;
\end{enumerate}
and any bounded series $\{q_t\}_{t \in [T]}$ such that $\max_{t \in [T]}|q_t| \le Q$, let $c_t =\sqrt{{{\tilde\kappa^{\dagger}_t}T}/{(n^{\dagger 2}_{\eff} z)}}$, we have
$$
    \PP\bigg[\frac1T \sum_{t=1}^T\{(\hat\sigma_e)_t^2 / \sigma_t^2 - 1\}q_t \le CQ \sqrt{\frac{z}{T}} + \frac1T \sum_{t=1}^T \frac{\delta_{0,t}q_t}{\sigma_t^2}\bigg] \ge 1- 2 e^{-z}\,,
$$
where $C$ depends on $\max_{t \in [T],u \in [m]} \{\tilde\kappa_{t+u} / \sigma_t^4\}$.
\end{thm}

\begin{rem}
Let $z = \log T$. 
To achieve the optimal rate of $\sqrt{\log T/T}$, Theorem \ref{thm4} requires that $\frac1T \sum_{t=1}^T \frac{\delta_{0,t}q_t}{\sigma_t^2}$ is of order $\sqrt{\log T/T}$. We also require $n^\dagger_{\eff}$ to be at least the order of $\sqrt{T}$.
\end{rem}

\begin{rem}
    One technical challenge of proving Theorem \ref{thm4} lies in the overlap of squared returns that are used to construct neighboring $\hat\sigma_t^2$, which leads to temporal dependence across $\{(\hat \sigma_e)^ 2_t\}_{t \in [T]}$. To resolve this issue, we apply a more sophisticated variant of Bernstein's inequality for time series data \citep{zhang2021robust}. See Lemma \ref{lem1} in the appendix.
\end{rem}

% \noindent {\bf \underline{Third option: achieving central moment}}

% These two benchmark volatility proxies are easy to analyze as we have closed form expressions for $\hat\sigma_t^2$. Note that the error bounds of both proxies rely on the absolute fourth moment of the return. 
We now move on to the Huber volatility proxy. At time $t$, denote the solution to \eqref{eq:tuning_free_huber_weighted} by $(\hat\theta_t, \hat\tau_t)$. Note that $\hat\tau_t$ is tuned based on just $m$ data points. Given that now our focus is on the average deviation of volatility proxies over $T$ data points, we need to raise our robustification parameters to reduce the bias of our Huber proxies and rebalance the bias and variance of the average deviation. After all, averaging over a large $T$ mitigates the impact of possible tail events, so that we can relax the thresholding effect of the Huber loss. Specifically, let $c_t = \hat\tau_t \sqrt{T/n^\dagger_{\eff}}$, which is of order $\sqrt{{{\kappa^{\dagger}_t}T} / {(n^{\dagger 2}_{\eff} z)}}$ according to Theorem 2. Then we substitute $\tau_t = c_t$ into the first equation of \eqref{eq:tuning_free_huber_weighted} and solve for $\sigma_t^2$ therein to obtain the adjusted proxy; that is to say, the final $(\hat\sigma_H)_t ^ 2$ satisfies the following: 
\beq
    \label{eq:huber_proxy}
    \sum_{s=t}^{t+m} w_{s,t} \min\bigg(|X_{s, t}^2 - (\hat\sigma_H)_t^2|, \frac{c_t}{w_{s,t}}\bigg) \sgn(X_{s, t}^2 - \sigma_t^2) = 0. 
\eeq
% Such inflation of the robustification parameter serves a similar purpose of circumventing overly underestimated proxies for volatility comparison. 
The inflation factor $\sqrt{T / n^\dagger_{\eff}}$ in $c_t$ implies that the larger sample size we have, the closer the corresponding Huber loss is to the square loss. This justifies the usage of vanilla EWMA proxy as the most common practice in financial industry when the total evaluation period is long. However, when $T$ is not sufficiently large, the Huber proxy yields more robust estimation of the true volatility.
The following theorem characterizes the average relative deviation of the Huber proxies.
\begin{thm}\label{thm5}
Under Assumption \ref{assump1}, for any $(c, z) >0$ satisfying that 
\begin{enumerate}
    \item $\max_{t \in [T]} \bigg\{ 2|\delta_{0,t}|/\sqrt{{\kappa^{\dagger}_t}} + 4\delta_{1,t} n^\dagger_{\eff}/({\kappa^{\dagger}_t}\sqrt{T})\bigg\} \le c$; 
    \item $n^\dagger_{\eff} \ge 2c^{-1}\{1 + (\log 2T)/z\} \sqrt{T}$; 
    \item $\sqrt{T} \ge 16 c z$, $T \ge 16z$;
    \item For any $t \in [T]$, $\cL_{c_t}(\theta; \{w_{s, t}\}_{s = t} ^ {t + m})$ is $\alpha$-strongly convex for $\theta \in \bigg(\sigma_t ^2 - (2c + 2)\sqrt{\kappa_t^{\dagger} z}, \sigma_t ^2 + (2c + 2)\sqrt{\kappa_t^{\dagger} z}\bigg)$;
\end{enumerate}
and any bounded series $\{q_t\}_{t \in [T]}$ such that $\max_{t \in [T]}|q_t| \le Q$, let $c_t =\sqrt{{{\kappa^{\dagger}_t}T}/(n^{\dagger 2}_{\eff} z)}$, we have
$$
\PP\bigg[\frac1T \sum_{t=1}^T\{(\hat\sigma_H)_t^2 / \sigma_t^2 - 1\}q_t \le C Q\sqrt{\frac zT} + \frac1T \sum_{t=1}^T \frac{(\delta_{0,t} + 4\delta_{1,t})q_t}{\sigma_t^2}\bigg] \ge 1- 2 e^{-z}\,,
$$
where $C$ depends on $\max_{t \in [T],u \in [m]} \{\kappa_{t+u} / \sigma_t^4\}$ and $\alpha$. 
\end{thm}
% Another thought is to get $\hat\tau_t$ from the iterative approach, and then define 
% With this series of $\{\hat\sigma_t^2\}$, are we able to get similar results as Theorem \ref{thm3} and Theorem \ref{thm4}, but this time the bound depends on central fourth moment. 

\begin{rem}\label{rem2}
Compared with the previous two benchmark proxies, the main advantage of the Huber volatility proxy is that its average deviation error depends on the central fourth moment of the returns instead of the absolute one. 
% Achieving central fourth moment in theory only affects constant terms. However, in practice this effect matters a lot. Think of the extreme situation when we have a very small clipping bound that always truncates data. With $c_t \le X_t^2$, the first proxy will give $\hat\sigma_t^2 = c_t$, and with $c_t \le \min_s w_{s,t} X_s^2$, the second option will give $\hat\sigma_t^2 = m c_t$. Both proxies can be arbitrarily small. In contrast, even with $c_t \le \min_s w_{s,t} |X_s^2 - \sigma_t^2|$, the third proxy solves $\sum_s c_t \sgn(X_s^2 - \sigma_t^2)$ giving $\hat\sigma_t^2 = \text{median}(X_s^2)$. This guarantees a desired protection for arbitrarily underestimated proxies. 
\end{rem}

\begin{rem}
$\alpha$-strong convexity is a standard assumption that can be verified for Huber loss. For example, Proposition B.1 of \cite{chen2020robust} shows that the equally weighted Huber loss $\cL_{\tau_t}(\theta; \{1/(m+1)\}_{s = t} ^ {t + m})$ enjoys strong convexity for $\alpha = 1/4$ in the region $(\sigma_t^2 - r, \sigma_t^2 + r)$ where $r \asymp \tau$ with probability at least $1-e^{-t}$ when $m \ge C(1+t)$. Such strong convexity paves the way to apply Lemma \ref{lem1} to the Huber proxies, so that we can establish their Bernstein-type concentration. Please refer to Lemma \ref{lem2} in the appendix for details. 
% Besides, our robustification parameter $c_t$ is large 
\end{rem}

\begin{rem}
To achieve the optimal $\sqrt{\log T/T}$ rate of convergence if we choose $z \asymp \log T$, we need to additionally assume $\frac1T \sum_{t=1}^T \frac{(\delta_{0,t} + 4\delta_{1,t})q_t}{\sigma_t^2}$ is of a smaller order, which in practice requires certain volatility smoothness.
\end{rem}

% We still like to prove the theorem using Lemma \ref{lem1}. However, now $\hat\sigma_t$ is not an explicit formula of the data once we solve the first-order equation. We need a perturbation bound for the solution of Huber loss minimization, whose details are given in . The proof of Theorem \ref{thm5} is also deferred to the appendix. 

\section{Robust predictors} \label{sec4}

In this section, we further take into account the randomness from predictors and study bounding \rnum{2}, the difference between empirical risk and conditional risk. 

%The previous section shows that $\PP_{\mathcal G_t}\big[\frac{1}{T} \sum_{t=1}^T \EE\{L(\hat\sigma_t^2, h_{2t})|\mathcal G_t\} - \EE\{L(\hat\sigma_t^2, h_{1t})|\mathcal G_t\} < 0\big]$ is small given $\frac{1}{T} \sum_{t=1}^T [\EE\{L(\sigma_t^2, h_{2t})\} - \EE\{L(\sigma_t^2, h_{1t})\}]$ is positive. In this section, we further take into account the randomness of predictors, so that we can finally compare the empirical risk of two predictors. Given the bounds for \rnum{1} with different volatility proxies in Theorems \ref{thm3}, \ref{thm4} and \ref{thm5}, we can now focus on concentration of the empirical risk to the conditional expected risk given proxies.
% , i.e., $\frac{1}{T}\sum_{t = 1} ^ T L(\hat\sigma_t ^ 2, h_t)$, from its expectation. 
% However, in practice, we can only have access to the average empirical loss $\frac{1}{T} \sum_{t=1}^T L(\hat\sigma_t^2, h_{it})$, instead of the average population loss $\frac{1}{T} \sum_{t=1}^T \EE[L(\hat\sigma_t^2, h_{it})|\mathcal G_t]$. This difference also involves the randomness from the predictors. 
% To bound the probability of making false conclusion with the average empirical loss, we need to make sure the predictors are well behaved with the presence of heavy-tailed returns. That is, for any proxy sets $\{\hat\sigma_t^2\}$, the following probability is small:

\subsection{Robust volatility predictor} \label{sec4.1}

We essentially follow \eqref{eq:huber_weighted}, the sample-weighted Huber mean estimator, to construct robust volatility predictors. The only difference is that now we cannot touch any data beyond time $t$; we can only look backward at time $t$. Consider the following volatility predictor based on the past $m$ data points with backward exponential-decay weights: 
\beq
    \label{eq:huber_weighted_predict}
    h_t := \argmin_{\theta \in \RR} \sum_{s = t - m} ^ {t - 1} \nu_{s, t}\ell_{{\tau_t / \nu_{s, t}}}(X_s ^ 2 - \theta) =: \mathcal L_{\tau_t}(\theta; \{\nu_{s, t}\}_{s = t - m} ^ {t - 1}). 
\eeq
Similarly to \eqref{eq:tuning_free_huber_weighted}, 
we iteratively solve the following two equations for $h_t$ and $\tau_t$ simultaneously: 
\beq
    \label{eq:tuning_free_huber_weighted_predict}
    \left\{
    \begin{aligned}
        & \sum_{s=t-m}^{t-1} \nu_{s,t} \min\bigg(|X_s^2 - h_t^2|, \frac{\tau_t}{\nu_{s,t}}\bigg) \sgn(X_s^2 - \sigma_t^2) = 0; \\
        & \sum_{s=t-m}^{t-1} \nu_{s,t}^2 \min\bigg(|X_s^2 - h_t^2|^2, \frac{\tau_t^2}{\nu_{s,t}^2}\bigg) /\tau_t^2 - z = 0,     
    \end{aligned}
    \right.
\eeq
where we recall that $\nu_{s,t} = \lambda^{t-1-s} / \sum_{s=t-m}^{t-1} \lambda^{t-1-s}$. 
Theorem \ref{thm1} showed concentration of $\hat\sigma_t^2$ around $\sigma_t^2$, i.e. the loss is MSE. More generally, we hope to give results on $\PP(L(\sigma_t^2, h_t) - \min_h L(\sigma_t^2, h_t) > \varepsilon)$. According to (a) in Section \ref{sec2.1}, $\min_h L(\sigma_t^2, h_t) = f(\sigma_t^2) + B(\sigma_t^2) +  C(\sigma_t^2) \sigma_t^2$. Therefore, we hope to bound 
$$
    \PP(|L(\sigma_t^2, h_t) - L(\sigma_t^2, \sigma_t^2)| > \varepsilon)
$$
This should be easy to control if we assume smoothness of the loss fucntion. We give the following theorem.

\begin{thm} \label{thm6}
Assume there exist $b, B$ such that $sup_{|h-\sigma_t^2| \le b} | \partial L(\sigma_t^2, h) / \partial h| \le B$. 
If $\tau_t =\sqrt{\frac{\kappa ^ {\ddagger}_t}{n^\ddagger_{\eff} z}}$, under Assumption \ref{assump1}, for $n^\ddagger_{\eff} \ge 16 (1 \vee \kappa ^ {\ddagger}_t / b^2) z$, we have
$$
\PP\bigg(L(\sigma_t^2, h_t) - \min_h L(\sigma_t^2, h)  \le 4 B \sqrt{\kappa ^ {\ddagger}_t z/n^\ddagger_{\eff}} \bigg) \ge 1-2e^{-z+ |\Delta_{0,t}| \sqrt{{n^\ddagger_{\eff} z}/{\kappa ^ {\ddagger}_t}} + 2\Delta_{1,t}{n^\ddagger_{\eff} z}/{\kappa ^ {\ddagger}_t} }
$$
where $\hat\sigma_t^2$ is the solution of the first equation.
\end{thm}

\begin{rem}
Here for notational simplicity, we used the same estimation horizon $m$, same exponential-decay $\lambda$ for constructing both predictors and proxies. But in practice, of course they do not necessarily need to be the same. 
In our real data example, we will use a slower decay for constructing predictors while using a faster decay for proxies, which seems to be the common practice for real financial data, where we typically use more data for constructing predictors and less data for constructing proxies. 
%This is also motivated by our first proxy construction that once a proper truncation is applied, even using single time clipping can be used for proxy construction. 
We will stick to $m$ equals twice the half-life so that equivalently, we use a longer window for predictors and shorter window for proxies. 
\end{rem}

\begin{rem} \label{rem8}
In addition, here we also simplify the theoretical results by assuming the same $z$ for constructing proxies and predictors. We do not need to use the same $z$ controlling the tail probability for predictors and proxies practically. 
For predictors, as we focus on local performance, it is more natural to use $z = \log n^\ddagger_{\eff}$ following \cite{wang2020new}. For proxies, as we focus on overall evaluation, for a given $T$ we can take $z = \log T$. Sometimes, we want to monitor the risk evaluation as $T$ grows, then a changing $z$ may not be a good choice; we do not want to re-solve the local weighted tuning-free Huber problem every time $T$ changes. Therefore, we recommend to use $z = C \log n^\dagger_{\eff}$ for a slightly larger $C$, e.g. $z = 2 \log n^\dagger_{\eff}$ as in our real data analysis. 
\end{rem}

\subsection{Concentration of robust and non-robust predictors} \label{sec4.2}

Recall a robust loss satisfies $\EE\{L(\hat\sigma_t^2, h_{it}) | \mathcal G_t\} = \EE[f(h_{it})] + B(\hat\sigma_t^2) + \EE[C(h_{it})] \hat\sigma_t^2$.
So we further bound \rnum{2} as follows: 
\begin{equation*}
\begin{aligned}
\text{\rnum{2}} & = \PP\bigg[\bigg|\frac{1}{T} \sum_{t=1}^T \{f(h_{t}) - \EE f(h_{t})\}  +  \frac{1}{T} \sum_{t=1}^T \{C(h_{t}) - \EE C(h_{t})\}\hat\sigma_t^2 \bigg| > \varepsilon \bigg]  \\
& \le \PP\bigg[\bigg|\frac{1}{T} \sum_{t=1}^T f(h_{t}) - \EE f(h_{t}) + \{C(h_{t}) - \EE C(h_{t})\}\sigma_t^2 \bigg| \\
& \qquad \qquad \qquad \qquad \qquad + \bigg| \frac{1}{T} \sum_{t=1}^T (C(h_{t}) - \EE C(h_{t}))(\hat\sigma_t^2 - \sigma_t^2) \bigg| > \varepsilon\bigg] \\
& \le \underbrace{\PP\bigg[\bigg|\frac{1}{T} \sum_{t=1}^T (L(\sigma_t^2, h_t) - \EE[L(\sigma_t^2, h_t)]) \bigg| > \frac{\varepsilon}{2}\bigg]}_{\Delta_A} + \underbrace{\PP\bigg[\bigg| \frac{1}{T} \sum_{t=1}^T (C(h_{t}) - \EE[C(h_{t})])(\hat\sigma_t^2 - \sigma_t^2) \bigg| > \frac{\varepsilon}{2} \bigg]}_{\Delta_B}\,.
\end{aligned}
\end{equation*}

We wish to show that both $\Delta_A$ and $\Delta_B$ can achieve the desired rate of  $\sqrt{\log T/T}$ for a broad class of predictors and loss functions.  When $h_t$ is the proposed robust predictor in Section \ref{sec4.1}, we can obtain sharp rates of $\Delta_A$ and $\Delta_B$ as expected. Moreover, for the vanilla (non-robust) EWMA predictor $h_t = \sum_{s=t-m}^{t-1} \nu_{s,t} X_s^2$, we are also able to obtain the same sharp rates for $\Delta_A$ and $\Delta_B$ for Lipschitz robust losses. The third option is to truncate the predictor, i.e. $h_t = (\sum_{s=t-m}^{t-1} \nu_{s,t} X_s^2) \wedge M$ with some large constant $M$ and we can control the two terms for general robust losses. The bottom line is that for non-robust predictors, we need to make sure the loss does not go too crazy either via shrinking predictor's effect on the loss (bounded Lipschitz) or clipping the predictor directly.

The interesting observation is that bounding \rnum{2}, or the difference between empirical risk and conditional risk, requires minimal assumption such that most of the reasonable predictors, say in the M-estimator form, have no problem satisfying the concentration bound with proper choice of a robust proxy, although we do require the loss not to become too wild (see Theorem 7 for details). 
% Then it is the second step that really matters to compare the expected losses of the predictors under consideration. 
Technically, the concentration here only cares about controlling variance and does not care about the bias between $\EE[L(\sigma_t^2, h_t)]$ and $L(\sigma_t^2, \sigma_t^2)$ and between $\EE[C(h_{t})]$ and $C(h_{t})$. There is no need to carefully choose the truncation threshold to balance variance and bias optimally.

\begin{thm}\label{thm7}
For any $t \in [T]$, $\cL_{\tau_t}(\theta; \{\nu_{s, t}\}_{s = t-m} ^ {t-1})$ is $\alpha$-strongly convex for $\theta \in (\sigma_t ^2/2, 3\sigma_t ^2/2)$.
Choose $\hat \sigma _t = (\hat\sigma_e)_t$ (or $(\hat\sigma_H)_t$) with robustification parameter $c_t$ specified in Theorem \ref{thm4} (or Theorem \ref{thm5}). Under the assumptions of Theorem \ref{thm4} (or Theorem \ref{thm5}),
% If using the first proxy, we additionally require $z \ge \tilde c(\log T)^2$. Under any of the following three sets of conditions,
the bound 
$$
\PP\bigg(\bigg|\frac{1}{T} \sum_{t=1}^T L(\hat\sigma_t^2, h_{t}) -  \frac{1}{T} \sum_{t=1}^T \bigg(\EE f(h_{t}) + B(\hat\sigma_t^2) + \EE C(h_{t}) \hat\sigma_t^2\bigg) \bigg| < C' \sqrt{z/T} \bigg) \ge 1 -  (2T+5) e^{-z/2}
$$
holds for all the following three cases:
\begin{enumerate}
    \item For the robust predictors proposed in Section \ref{sec4.1}, assume $sup_{|h-\sigma_t^2| \le b} | \partial L(\sigma_t^2, h) / \partial h| \le B_t(b)$, if $|\Delta_{0,t}| \sqrt{{n^\ddagger_{\eff}}/{\kappa ^ {\ddagger}_t}} + 2\Delta_{1,t}{n^\ddagger_{\eff}}/{\kappa ^ {\ddagger}_t} < 1/2$ and $n^\ddagger_{\eff} \ge 64 z \max_t \kappa ^ {\ddagger}_t/\sigma_t^4$, the above bound holds with $C'$ depending on $\max_t B_t(\sigma_t^2/2)$ and C in Theorem \ref{thm4} (or Theorem \ref{thm5}).
    
    \item For the clipped vanilla non-robust exponentially weighted moving average predictors $h_t = (\sum_{s=t-m}^{t-1} \nu_{s,t} X_s^2) \wedge M$, assume $sup_{\sigma_t^2/2 \le h \le b} | \partial L(\sigma_t^2, h) / \partial h| \le B_t(b)$, if $|\Delta_{0,t}| \sqrt{{n^\ddagger_{\eff}}/{{\tilde\kappa^{\ddagger}_t}}} + 2\Delta_{1,t}{n^\ddagger_{\eff}}/{{\tilde\kappa^{\ddagger}_t}} < 1/2$ and $n^\ddagger_{\eff} \ge 64 z \max_t {\tilde\kappa^{\ddagger}_t}/\sigma_t^4$, the above bound holds with $C'$ depending on $\max_t B_t(M)$ and $C$ in Theorem \ref{thm4} (or Theorem \ref{thm5}).
    
    \item For the vanilla non-robust exponentially weighted moving average predictors $h_t = \sum_{s=t-m}^{t-1} \nu_{s,t} X_s^2$, assume $|\partial L(\sigma_t^2, h) / \partial h| \le B_0$ and $|L(\sigma_t^2, h)| \le M_0$ for $\forall \sigma_t^2, h$, if $|\Delta_{0,t}|\allowbreak \sqrt{{n^\ddagger_{\eff}}/{{\tilde\kappa^{\ddagger}_t}}} + 2\Delta_{1,t}{n^\ddagger_{\eff}}/{{\tilde\kappa^{\ddagger}_t}} < 1/2$ and $n^\ddagger_{\eff} \ge 64 z \max_t {\tilde\kappa^{\ddagger}_t}/\sigma_t^4$, the above bound holds with $C'$ depending on $B_0, M_0$ and $C$ in Theorem \ref{thm4} (or Theorem \ref{thm5}).
\end{enumerate}
\end{thm}

\begin{rem}
The proof can be easily extended to more general robust or non-robust predictors of the M-estimator form. Theorem \ref{thm7} tells us that comparing robust predictors with optimal truncation for a single time point (rather than adjusting the truncation as in constructing proxies) and non-robust predictors (either with rough overall truncation when using a general loss or without any truncation when using a truncated loss) is indeed a valid thing to do, when we employ proper robust proxies. 
\end{rem}

\begin{rem}\label{rem10}
Although the first proxy achieves the optimal rate of convergence for comparing average conditional loss in Theorem \ref{thm3}, we did not manage to show it is valid for comparing the average empirical loss in Theorem \ref{thm7}.
The reason is that single time clipping has no concentration guarantee like Theorem \ref{thm1} for a single time point, therefore cannot ensure $|(\hat\sigma_c)_t^2 - \sigma_t^2|$ to be bounded with high probability for all $t$, which is important to make sure the sub-exponential tail in Bernstein inequality does not dominate the sub-Gaussian tail. 
%This additional condition comes from Lemma \ref{lem1}, where the sub-exponential tail part of the Bernstein-type inequality has an extra $(\log T)^2$ in the bound. If a better bound without this artificial $(\log T)^2$ term is possible, then this can be removed. 
%However, this shows a potential limitation of using the first proxy: it is possible for the other two proxies to conclude $|\hat\sigma_t^2 - \sigma_t^2| \le C\sqrt{z}$ with high probability for $z \asymp \log T$, leveraging neighborhood volatility smoothness to get good concentration; but if we want $|\hat\sigma_t^2 - \sigma_t^2| \le C\sqrt{zT}/(\log T)^2$ for the first proxy to use Lemma \ref{lem1}, we can only do so by enforcing a large $z \ge \tilde c (\log T)^2$. Then the deviation bound in Theorem \ref{thm7} would not achieve the optimal convergence order of $\sqrt{\log T/T}$. 
Taking into consideration of central fourth moment versus absolute fourth moment (see Remark \ref{rem2}), we would recommend the third proxy as the best practical choice among our three proposals.
\end{rem}

\section{Numerical study} \label{sec5}

In this section, we first verify the advantage of the Huber mean estimator over the truncated mean in terms of estimating the variance through simulations. As illustrated by Theorem \ref{thm5}, the statistical error of the Huber mean estimator depends on the central moment, while that of the truncated mean depends on the absolute moment. Then we apply the proposed robust proxies for volatility forecasting comparison, using data from crypto currency market. Specifically, we focus on the returns of Bitcoin (BTC) quoted by Tether (USDT), a stable coin pegged to the US Dollar, in the years of 2019 and 2020, which witness dramatic volatility of Bitcoin.
% . Trading a single product like BTC can be very volatile, compared to other more mature financial instruments, like equity or macro products. 

\subsection{Simulations}

We first examine numerically the finite sample performance of the adaptive Huber estimator \citep{wang2020new} for variance estimation, i.e., we solve  \eqref{eq:tuning_free_huber} iteratively for $\theta$ given the data and $z$ until convergence. We first draw an independent sample $Y_1, \dots, Y_n$ of $Y$ that follows a heavy-tailed distribution. We investigate the following two distributions: 
\begin{enumerate}
\item Log-normal distribution $\mathrm{LN}(0, 1)$, that is, $\log Y \sim \cN(0,1)$.
\item Student's t distribution with degree of freedom $\mathrm{df} = 3$.
\end{enumerate}
Given that $\sigma^2 := \Var(Y) = \EE(Y^2) - (\EE Y)^2$, we estimate $\EE (Y^2)$ and $\EE Y$ separately and plug these mean estimators into the variance formula to estimate $\sigma^2$. Besides the Huber mean estimator, we investigate two alternative mean estimators as benchmarks: (a) the na\"ive sample mean; (b) the sample mean of data that are truncated at their upper and lower $\alpha$-percentile. We use MSE and QL (see \eqref{eq:mse_ql} for the definitions) to assess the accuracy of variance estimation. In our simulation, we set $n = 100$ and we repeat evaluating these three methods in 2000 independent Monte Carlo experiments. 

\begin{figure}
    \includegraphics[width=.5\columnwidth]{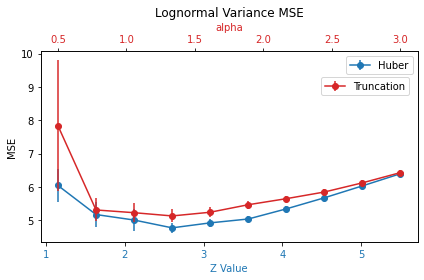}\hfill
     \includegraphics[width=.5\columnwidth]{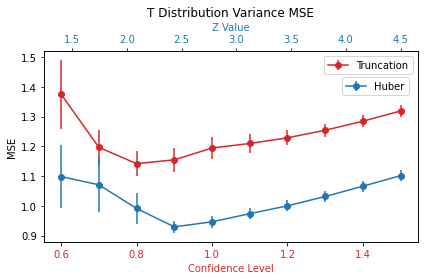}
    \caption{MSE comparisons between the truncated mean and the Huber mean.}
    \label{fig1}
\end{figure}    
% Note that the goal of this section is neither to evaluate robust proxies since here we obviously know the truth, nor to test our extension of unequally weighted tuning-free Huber minimization. Nonetheless, this section is meant to investigate the robustness of the original Huber minimization proposed in \cite{wang2020new} to the deviation parameter $z$ and compare Huber minimization to the more commonly used truncation approach, where we hope to conclude the advantage of the former using the central fourth moment in its deviation bound. We will see that the loss of Huber minimization is rather flat and not sensitive to the selection of $z$, compared to the truncation method, which is slightly more sensitive to the selection of $\alpha$. More importantly, the constant improvement from absolute fourth moment in the rate of the truncation method to central fourth moment in the rate of the Huber minimization matters as their loss curves differ by a clear margin.
% For each distribution, let $n = 100$ and Again we hope to emphasize, ``tuning-free'' here means the method is adaptive to unknown variance, but we still have a tuning parameter $z$. This is particularly important for rolling volatility forecasting, because we can fix the same $z$ for the forecasting at each time $t$, and the system will adapt to the time-varying volatility automatically. And tuning for $z$ is both less sensitive and more intuitive as $z$ controls the level of tail probability.

Figure \ref{fig1} compares the MSE of the truncated and Huber variance estimators under log-normal distribution (left) and t-distribution (right). The red curve represents the MSE of the truncated method with different $\alpha$ values on the top x-axis, and the blue curve represents the MSE of the tuning-free Huber method with different $z$ values on the bottom x-axis. The error bars in both panels represent the standard errors of the MSE. For convenience of comparison, we focus on the ranges of $\alpha$ and $z$ that exhibit the smile shapes of MSE of the two methods. 
%In graphs of these two heavy-tailed distributions, the DA-Huber estimator and truncated estimator form similar patterns in the stated intervals.  
Note that the MSE of the sample variance, which is $40.14$ under the log-normal distribution and $10.17$ under the t-distribution, is too large to be presented in the plot. Figure \ref{fig1} shows that the Huber variance estimator outperforms the optimally tuned truncated method with $z$ roughly between $(1.5, 3.5)$ under the log-normal distribution and between $(1.5, 4)$ under the Student's t distribution. The performance gap is particularly large under the t distribution, where the optimal Huber method achieves around $20\%$ less MSE than the optimal truncated method.
% Besides, the standard errors exhibit a decreasing trend with the growth in $\alpha$ and deviation tolerance level $z$.

\begin{figure}
    \includegraphics[width=.5\columnwidth]{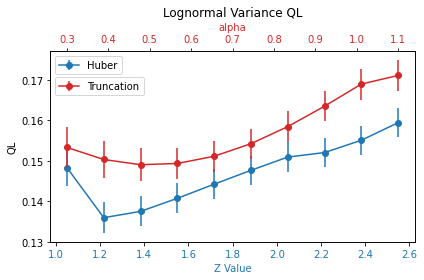}\hfill
     \includegraphics[width=.5\columnwidth]{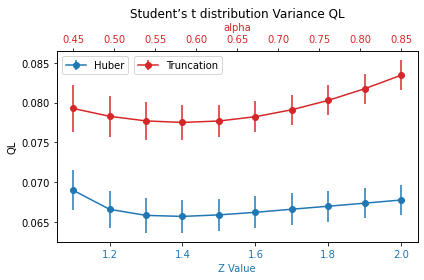}
    \caption{QL comparisons between truncation and Huber minimization.}
    \label{fig2}
\end{figure}  

Figure \ref{fig2} assesses the QL loss of the truncated and Huber estimators and displays a similar pattern as Figure \ref{fig1}.  
% The conclusions are all similar to Fig \ref{fig1}, except that the error bars are more stable with the selection of $\alpha$ and $z$.
%Different from the MSE case, the standard errors of QL almost keep constant throughout the confidence level and Z-value intervals in both distributions.
Again, the na\"ive sample variance is still much worse off with QL loss $0.1765$ under the log-normal distribution and $0.1201$ under the t-distribution; we therefore do not present it in the plots. The Huber approach continues to defeat the optimally tuned truncation method with any $z \in (1, 2)$ under both distributions of our focus. Together with the $z$ ranges where the Huber approach is superior in terms of MSE, our results suggest that $z = 1.5$ can be a good practical choice, at least to start with. Such a universal practical choice of $z$ demonstrates the adaptivity of the tuning-free Huber method. 
% It is natural to see similar results using MSE or QL when we have only a single quantity to estimate. We will see the practical volatility forecasting at multiple times can get more tricky in terms of scaling.

\subsection{BTC/USDT volatility forecasting}

% Real data can be hard to simulate, so we directly work with the real data. 
We use the BTC/USDT daily returns to demonstrate the benefit of using robust proxies in volatility forecasting comparison.
\begin{figure}
    \includegraphics[width=\columnwidth]
        {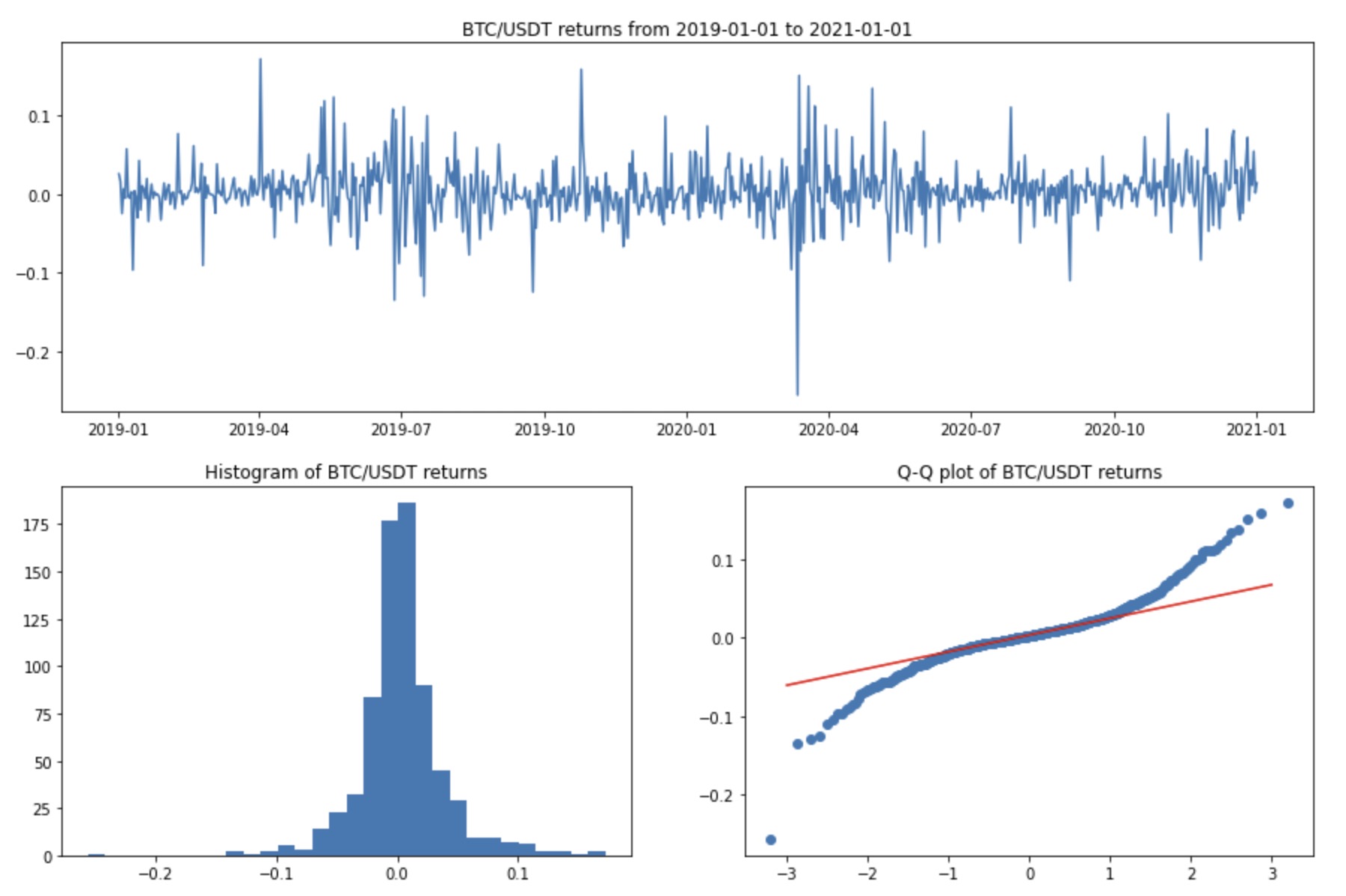}
    \caption{Time series, histogram and Q-Q plot of the BTC/USDT returns.}
    \label{fig3}
\end{figure} 
Fig \ref{fig3} presents the time series, histogram and normal QQ-plot of the daily BTC returns from 2019-01-01 to 2021-01-01. It is clear that the distribution of the returns is heavy-tailed, that the volatility is clustered and that there are extreme daily returns beyond 10\% or even 20\%. The empirical mean of the returns over this 2-year period is 38 basis points, which is quite close to zero compared with its volatility. We thus assume the population mean of the return is zero, so that the variance of the return boils down to the mean of the squared return. In the sequel, we focus on robust estimation of the mean of the squared returns. 

\subsubsection{Construction of volatility predictors and proxies}

Let $r_t$ denote the daily return of BTC from the end of day $t-1$ to the end of day $t$. We emphasize that a volatility predictor $h_t$ must be ex-ante. Here we construct $h_t$ based on $r_{t-m_b},\dots,r_{t-1}$ in the backward window of size $m_b$ and evaluate it at the end of day $t-1$. Our proxy $\hat\sigma_t^2$ for the unobserved variance $\sigma_t^2$ of $r_t$ is instead based on $r_{t},\dots,r_{t+m_f}$ in the forward window of size $m_f$.

We consider two volatility prediction approaches: (i) the vanilla EWA of the backward squared returns, i.e., $\sum_{s=t-m_b}^{t-1} \nu_{s,t} X_s^2$; (ii) the exponentially weighted Huber predictor proposed in Section \ref{sec4.1}. Each approach is evaluated with half lives equal to $7$ days (1 weeks) and $14$ days (2 weeks), giving rise to four predictors, which are referred to as EWMA\_HL7, EWMA\_HL14, Huber\_HL7, Huber\_HL14. We always choose $m_b$ to be twice the corresponding half life and set $z = n^\ddagger_{\eff}$ for the two 'Huber predictors. As for volatility proxies, we similarly consider two methods: (i) the vanilla forward EWA proxy, i.e., $\sum_{s=t}^{t + m_f} w_{s,t} X_s^2$; (ii) the robust Huber proxy proposed in Section \ref{sec3.3}. We set the half life of the exponential decay weights to be always $7$ days, $m_f = 14$ and $z = 2\log n^\dagger_{\eff}$. We evaluate the Huber approach on two time series of different lengths: $T = 720$ or $180$, which imply two different $c_t$ values that are used in \eqref{eq:huber_proxy}. We refer to the two corresponding Huber proxies as Huber\_$720$ and Huber\_$180$. 
% Here truncation rescaling with $T = 720$ represents a situation when we have a longer time period to evaluate the results, so we are more confident and with a larger truncation level, we truncate less. Whereas when truncation rescaling is $T=180$, we end up a smaller truncation level and truncate more.
Given the theoretical advantages of the Huber proxy as demonstrated in Remarks \ref{rem2} and \ref{rem10}, we do not investigate the first two proxies proposed in Section \ref{sec3.3}. 
% We refer the readers to Remark \ref{rem8} for some justifications on this choice. 

Crypto currency market is traded 24 hours every day nonstop, which gives us $732$ daily returns from 2019-01-01 to 2021-01-01. After removing the first $27$ days used for predictor priming and the last $13$ days used for proxy priming, we then have $691$ data points left. For each day, we compute the four predictors and three proxies that are previously described. We plot the series of squared volatility (variance) proxies in Fig \ref{fig4}. As we can see, the vanilla EWMA proxy (blue line) is obviously the most volatile one, reaching the peak variance of $0.016$, or equivalently, volatility of $12.6\%$ in March, 2020, when the outbreak of COVID-19 in the US sparked a flash crash of the crypto market. In contrast, the Huber proxies react in a much milder manner, and the smaller $T$ we consider, the more truncation effect on the Huber proxies.

\begin{figure}
     \includegraphics[width=\columnwidth]
        {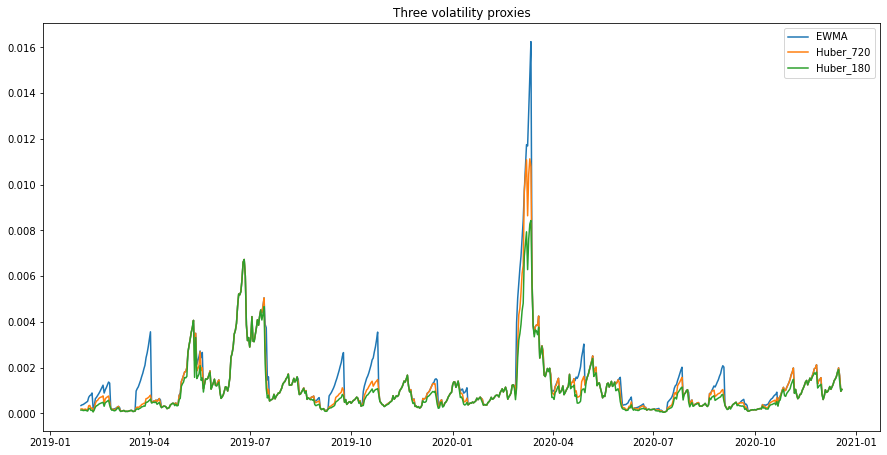}
    \caption{Three volatility proxies: non-robust EWMA proxy (blue) and the robust Huber proxies Huber\_$720$ (orange) and Huber\_$180$ (green). The plotted values are variances.}
    \label{fig4}
\end{figure}  

\subsubsection{Volatility forecasting comparison with large $T$}

With the predictors and proxies computed, we are ready to conduct volatility prediction evaluation and comparison.
% The most direct way is to compute the empirical risk over certain period of time and then compare the risk of different predictors. 
Now we would like to emphasize one issue that is crucial to the evaluation procedure: the global scale of the predictors. 
Different loss functions may prefer different global scales of volatility forecasts. 
For example, QL penalizes underestimation much more than overestimation, as the predictor is in the denominator in the formula of QL. In other words, QL typically favors relatively high forecast values. To remove the impact of the scales and focus more on the capability of capturing relative variation of volatility, we also compute optimally scaled versions of our predictors and evaluate their empirical loss. Specifically, we first seek the optimal scale by solving for
\[
    \hat\beta := \argmin_{\beta \in \RR} T^{-1}\sum_{t=1}^T L(\hat\sigma_t^2, \beta h_{t})
\]
and then use $\{\hat\beta h_t\}_{t \in [T]}$ for prediction. By comparing the empirical risk of the optimally scaled predictors, we can completely eliminate the discrimination of the loss against different global scales. Some algebra yields that for MSE, the optimal $\hat\beta_{\mse} = \sum_t h_t \hat\sigma_t^2 / \sum_t h_t^2$, and for QL, the optimal $\hat\beta_{\ql} = T^{-1} \sum_t (\hat\sigma_t^2 / h_t)$. Table \ref{tab1} reports the loss of the four predictors and their optimal scaled versions based on all the 691 time points with the non-robust EWMA proxy and the robust proxy Huber\_$720$. Several interesting observations are in order. 
% This is meaningful because one may not expect to purely use return information to detect volatility regime switching, so when we have the ability to adjust risk aversion levels according to extra information to predict the overall scaling within the certain period we investigate, a predictor that can best capture the relative ups and downs beyond the scaling can be favorable. The above discussions will become more clear next when we look at the real data.

% Given the tricky issue, we hope to emphasize two things. Firstly, different losses can lead to different comparison results. Robust loss of \cite{patton2011volatility} only says for a given robust loss, comparisons are consistent with respect to the choice of unbiased proxies. But comparisons can be inconsistent for two different robust losses. Therefore, we do not intend to argue which loss is better. Actually, different losses can serve different purposes. For volatility estimation, one may care about MSE, while for portfolio management where underestimation may lead to crazy risk exposure, one may care more about QL. We leave the choice of losses to investment managers who set their own investment goals and objectives. Secondly, using losses with optimal predictor scaling can also be meaningful. That is, 

\begin{table}[] 
\caption{Losses of (original or optimally scaled) four predictors with robust and non-robust proxies, evaluated with all $T = 691$ time points.}
\label{tab1}
\begin{tabular}{ccccccc}
\hline
\hline
                     & \multicolumn{6}{c}{MSE}                                                   \\ \cline{2-7} 
                     & \multicolumn{3}{c}{EWMA Proxy}     & \multicolumn{3}{c}{Huber\_720 Proxy} \\ \cline{2-7} 
                     & Orig (1e-6) & Scaled (1e-6) & $\hat\beta_{\mse}$ & Orig (1e-6)  & Scaled (1e-6)  & $\hat\beta_{\mse}$ \\ \hline
EWMA\_HL14           & 4.115       & 3.365         & 0.55 & 3.285        & 2.386          & 0.50 \\
Huber\_HL14          & 3.162       & 3.161         & 1.03 & 2.233        & 2.228          & 0.94 \\
EWMA\_HL7            & 4.824       & 3.395         & 0.46 & 3.930        & 2.364         & 0.44 \\
Huber\_HL7           & 3.112       & 3.110         & 1.05 & 2.134        & 2.133         & 0.98 \\ \hline \hline
\multicolumn{1}{l}{} & \multicolumn{6}{c}{QL}                                                    \\ \cline{2-7} 
\multicolumn{1}{l}{} & \multicolumn{3}{c}{EWMA Proxy}     & \multicolumn{3}{c}{Huber\_720 Proxy} \\ \cline{2-7} 
\multicolumn{1}{l}{} & Orig        & Scaled        & $\hat\beta_{\ql}$ & Orig         & Scaled         & $\hat\beta_{\ql}$ \\ \hline
EWMA\_HL14           & 0.804       & 0.647         & 1.67 & 0.584        & 0.548          & 1.29 \\
Huber\_HL14          & 1.352       & 0.567         & 2.82 & 0.831        & 0.450          & 2.14 \\
EWMA\_HL7            & 1.239       & 0.792         & 2.26 & 0.720        & 0.595          & 1.59 \\
Huber\_HL7           & 2.382       & 0.702         & 4.09 & 1.396        & 0.532          & 2.94 \\ \hline
\hline
\end{tabular}
\end{table}
% Now first focus on the regime of large all the data we have and the losses of MSE and QL. To get the optimal scaling for those losses, it is easy to optimize $\beta$ and find that 

% Note that since our total sample size is about $T \approx 720$, it makes no sense to use Huber\_$180$ for evaluation with the entire data. 

\begin{itemize}
    \item Using the longer half life of $14$ days gives a smaller QL loss, regardless of whether the predictor is robust or non-robust, original or optimally scaled, and regardless of whether the proxy is robust or non-robust. In terms of MSE, the half-life comparison is mixed: Huber\_HL7 is slightly better than Huber\_HL14, but EWMA\_HL14 is better than EWMA\_HL7. 
    % In terms of the bias-variance tradeoff of MSE, it seems EWMA has a huge variance, so that we have to bring it down using more effective samples or a slower decay. Once we control extreme tail events robustly, we can apply a faster decay to reduce the approximation bias from time-varying volatilities.
    We only focus on the longer half-life from now on. 
    \item If we look at the original predictors without optimal scaling, it is clear that MSE favors the robust predictor and QL favors the non-robust predictor, regardless of using robust or non-robust proxies. This confirms that different loss functions can lead to very different comparison results. 
    
    \item However, the above inconsistency between MSE and QL is mostly due to scaling, which is clearly demonstrated by the column of the optimal scaling $\hat\beta$. 
    % Huber predictor gives a smaller value compared to EWMA predictor. MSE which favors shrinkage will surely favor Huber predictor, whereas QL which penalizes underestimation will surely favor EWMA predictor. 
    For MSE, the optimal scaling of the EWMA predictor is around $0.5$, while that of the Huber predictor is around $1$. In contrast, for QL, the optimal scaling needs to be much larger than $1.0$ and Huber needs a even larger scaling. If we look at the loss function values with optimally scaled predictors, it is interesting to see that the Huber predictor outperforms the EWMA predictor in terms of both MSE (slightly) and QL (substantially). This means that the Huber predictor is more capable of capturing the relative change of time-varying volatility than the non-robust predictor. 
    
    \item Last but not least, when the sample size $T$ is large compared with $n^\dagger_{\eff}$ (here $T/n^\dagger_{\eff} = 691/24.87 = 27.78$), the difference between the EWMA and Huber proxies is small, which explains the reason they give consistent comparison results. When $T$ is not large enough in the next subsection, we can see that the robust proxies gives more sensible conclusions. 
\end{itemize}

\subsubsection{Volatility forecasting comparison with small $T$}

\begin{figure}
     \includegraphics[width=\columnwidth]
        {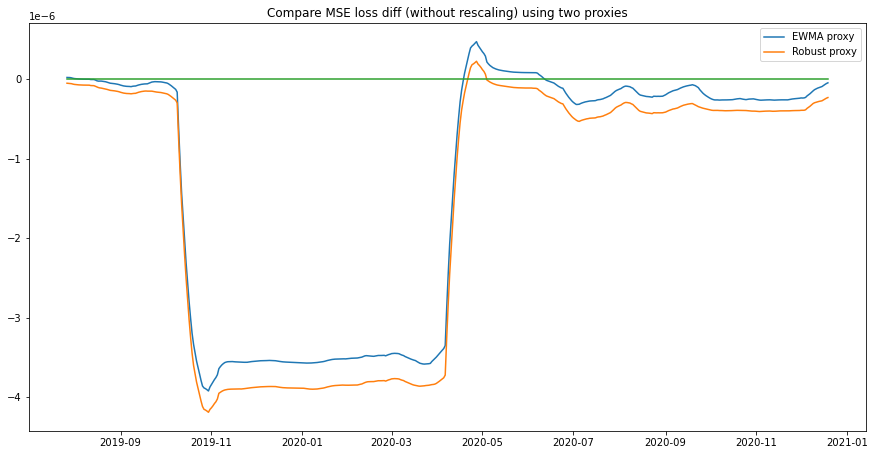}
     \includegraphics[width=\columnwidth]
        {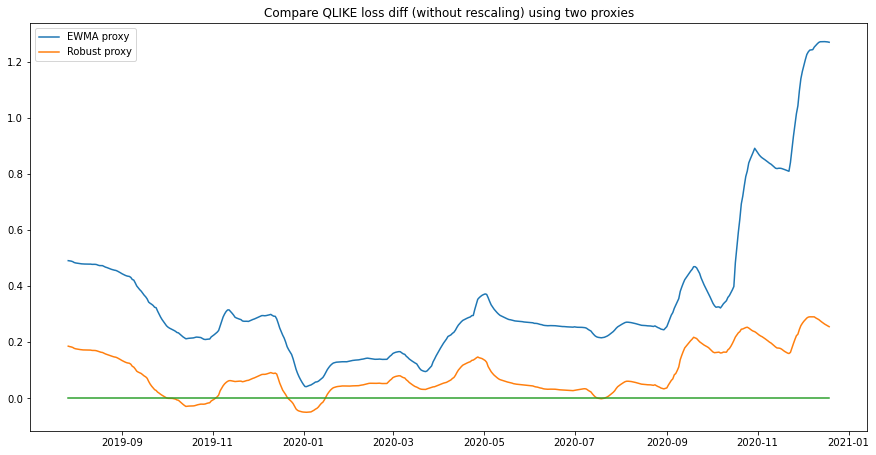}
    \caption{180-day rolling loss difference between EWMA\_HL14 and Huber\_HL14 with robust or non-robust proxies. The upper panel corresponds to MSE and the lower one corresponds to QL.}
    \label{fig5}
\end{figure}  

\begin{figure}
     \includegraphics[width=\columnwidth]
        {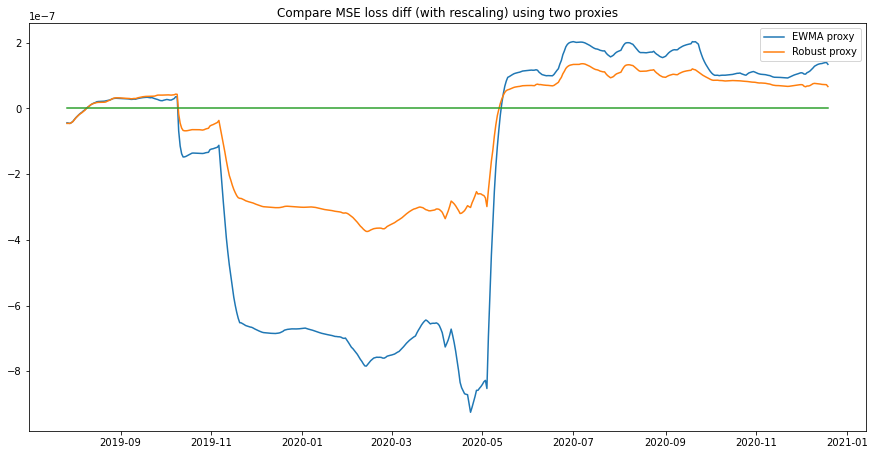}
     \includegraphics[width=\columnwidth]
        {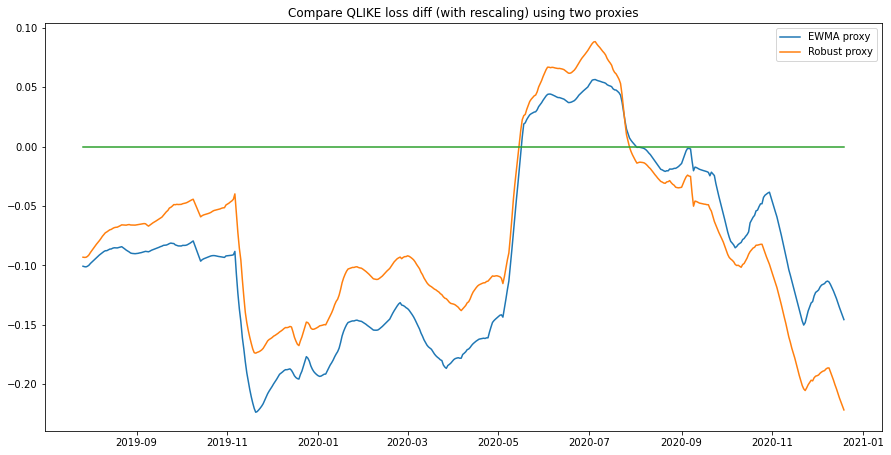}
    \caption{180-day rolling loss difference between optimally scaleda EWMA\_HL14 and Huber\_HL14 with robust or non-robust proxies. The upper panel corresponds to MSE and the lower one corresponds to QL.}
    \label{fig6}
\end{figure}  

\begin{figure}
     \includegraphics[width=\columnwidth]
        {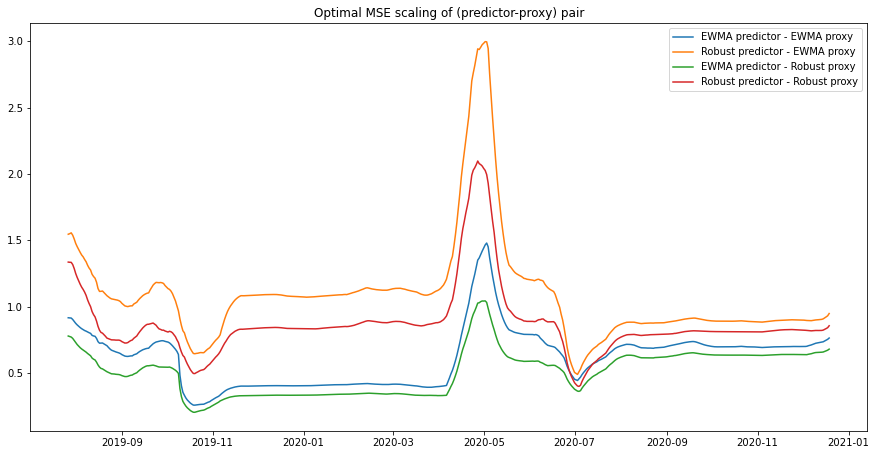}
     \includegraphics[width=\columnwidth]
        {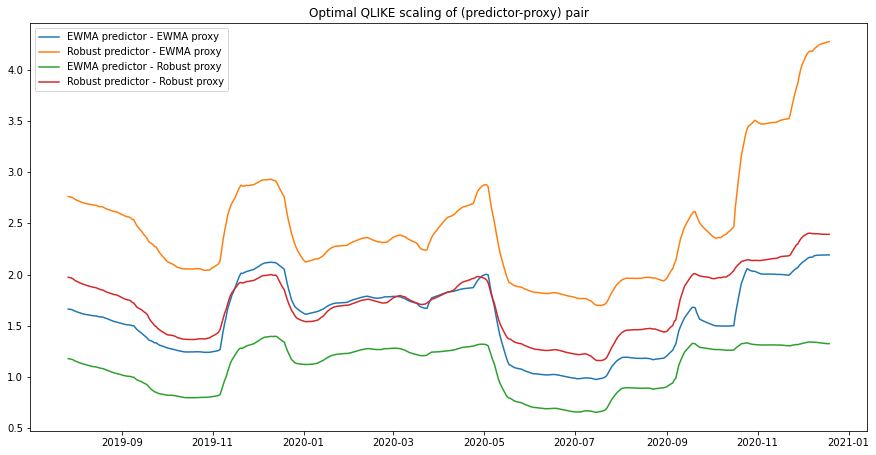}
    \caption{Optimal scaling of robust and non-robust predictors with robust and non-robust proxies. The upper panel is MSE and the lower one is QL.}
    \label{fig7}
\end{figure}  

Now suppose we only have $T=180$ data points to evaluate and compare volatility forecasts. In Fig \ref{fig5}, we present the curve of 180-day rolling loss difference, i.e., $T^{-1}\sum_{s=t-179}^t \allowbreak\{L(\hat\sigma_s^2, h_{\text{Huber\_HL14},s})  - L(\hat\sigma_s^2, h_{\text{EWMA\_HL14},s})\}$ with $t$ ranging from $180$ to $691$, where $\hat\sigma_s^2$ can either be the EWMA proxy or Huber\_$180$. Positive loss difference at $t$ indicates that the EWMA predictor outperforms the Huber predictor in the past 180 days. 
% This matches with our previous discussion on the tricky scaling issue. 
We see that most of the time, Huber\_HL14 defeats EWMA\_HL14 (negative loss difference) in terms of MSE, while EWMA\_HL14 defeats Huber\_HL14 (positive loss difference) in terms of QL. In terms of MSE, robust proxies tend to yield more consistent comparison between the two predictors throughout the entire period of time. We can see from the upper panel of Figure \ref{fig5} that the time period for EWMA\_HL14 to outperform Huber\_HL14 is much shorter with the robust proxies (orange curve) than that with the EWMA proxies (blue curve). 
In terms of QL, if we use the EWMA proxy, we can see from the lower panel of Figure \ref{fig5} that the robust predictor is much worse than the non-robust predictor, especially towards the end of 2020. However, the small MSE difference at the end of 2020 suggests that the EWMA proxy should overestimate the true volatility and exaggerate the performance gap in terms of QL. With the Huber proxy, however, the loss gap between the two predictors is much narrower, suggesting that the Huber proxy is more robust against huge volatility. 
% Robust proxy corrects the over-penalization of volatility underestimation when the non-robust proxy can easily give us an overestimated target value in a crisis volatile time period.

Fig \ref{fig6} presents the curve of 180-day rolling loss difference between optimally scaled \text{Huber\_HL14} and EWMA\_HL14, based on robust and EWMA proxies respectively. For MSE, from the previous subsection, we know that when the optimal scaling is applied, two predictors do not differ much in terms of the overall loss, and the Huber predictor only slightly outperforms the EWMA predictor. In the upper panel of Fig \ref{fig6}, we see that the robust-proxy-based curve is closer to zero than the EWMA-proxy-based curve, displaying more consistency with our result based on large $T$. For QL, the loss differences using the robust or the non-robust proxy look quite similar. We also plot $\hat\beta_{\text{Huber},t}$ and $\hat\beta_{\text{EWMA},t}$ versus time $t$ based on robust and non-robust proxies in Fig \ref{fig7}. For both MSE and QL losses, using the robust proxy leads to more stable optimal scaling values, which are always preferred by practitioners. 

In a nutshell, we have seen that how our proposed robust Huber proxy can lead to better interpretability and sensible comparison of volatility predictors. When total sample size is small compared to the local effective sample size, using a robust proxy is necessary and leads to a smaller probability of misleading forecast evaluation and comparison. When the total sample size is large enough, the proposed robust proxy automatically truncates less and resembles the EWMA proxy. This also provides justification for using a non-robust EWMA proxy when the sample size is large. But we still recommend the proposed robust proxy that can adapt to the sample size and the time-varying volatilities. Sometimes, even if the robust proxy only truncates data for a very small volatile period, in terms of risk evaluation, it could still cause significant difference.

\section{Discussions} \label{sec6}

Compared with the literature of modeling volatility and predicting volatility, evaluating volatility prediction has not been given enough attention. Part of the reason is due to lacking a good framework for its study, which makes practical volatility forecast comparison quite subjective and less systematic in terms of loss selection and proxy selection. \cite{patton2011volatility} is a pioneering work to provide one framework based on long term expectation and provides guidance for loss selection, while our work gives a new framework based on an empirical deviation perspective and further provides guidance on proxy selection. In our framework, we focus on predictors that can achieve desired probability bound for \rnum{2} so that the empirical loss is close to the conditional expected loss. Then the correct comparison of the conditional expected loss of two predictors with large probability rely on a good control for \rnum{1}, which imposes requirements for proxies. 

With this framework, we proposed three robust proxies, each of which guarantees a good bound for \rnum{1} when the data only bears finite fourth moment. Among the three proxies, although they can all obtain the optimal rate of convergence for bounding \rnum{1}, we recommend the exponentially-weighted tuning-free Huber proxy. It is better than the clipped squared returns in that it leverages neighborhood volatility smoothness and it is better than the proxy based on direct truncation in that its improved constant in the deviation bound, depending only on the central moment. To construct this proxy, we need to solve the exponentially-weighted Huber loss, whose truncation level for each sample also needs to change with the weights surprisingly.

We then applied this proxy to the real BTC volatility forecasting comparison and reached some interesting observations. Firstly, robust predictors with better control in variance may use a faster decay to reduce the approximation bias. Secondly, different losses can lead to drastically different comparison, so even restricting to robust losses, loss selection is still a meaningful topic in practice. Thirdly, predictor rescaling according to the loss function is necessary and could further extract the value of robust predictors. Finally, the proposed robust Huber proxy adapts to both the time-varying volatility and the total sample size. When the overall sample size is much larger than the local effective sample size, the robust Huber proxy barely truncate, which provides justification for even using the EWMA proxy for prediction evaluation. However, the robust Huber proxy in theory still gives high probability of concluding the correct comparison. 

There are still limitations of the current work and open questions to be addressed. Assumption \ref{assump1} excludes the situation when $\sigma_t^2$ depends on previous returns such as in GARCH models. 
We require $n^\dagger_{\eff}\to \infty$ for local performance guarantee, leading to a potentially slow decay. However, in practice, it is hard to know how fast the volatility changes. Also, we ignored the auto-correlation of returns and assumed temporal independence of the innovations for simplicity. Extensions of the current framework to time series models and more relaxed assumptions are of practical value to investment managers and financial analysts. Our framework may have nontrivial implications on how to conduct cross-validation with heavy-tailed data, where we use validation data to construct proxies for the unknown variable to be estimated robustly. Obviously, subjectively choosing a truncation level for proxy construction could favor certain truncation level used by a robust predictor. Motivated by our study, rescaling the optimal truncation level for one data splitting according to the total (effective) number of sample splitting sounds an interesting idea worth further investigation in the future.

\bibliographystyle{ims}
\bibliography{RobustVolEstAndEval}

\newpage	
\appendix
\section{Proofs}
	
This section provides proof details for all the theorems in the main text.
	
\begin{proof}[{\bf Proof of Theorem \ref{thm1}}]
The proof follows Theorem 5 of \cite{FLW16}. Denote $\phi(x) = \min(|x|, 1) \sgn(x)$, we have 
$$
-\log(1-x+x^2) \le \phi(x) \le \log(1+x+x^2)\,.
$$
Define $r(\theta) = \sum_{s=t}^{t+m} \min(w_{s,t}|X_s^2 - \theta|, \tau_t) \sgn(X_s^2 - \theta)$, so $\hat\sigma_t^2$ is the solution to $r(\theta) = 0$. 
\begin{equation*}
\begin{aligned}
\EE\{\exp[r(\theta) / \tau_t]\} & \le \prod_{s=t}^{t+m} \EE\{\exp[\phi(w_{s,t}(X_s^2 - \theta) / \tau_t)]\} \\
& \le \prod_{s=t}^{t+m} \EE\{1 + w_{s,t}(X_s^2 - \theta) / \tau_t + w_{s,t}^2(X_s^2 - \theta)^2 / \tau_t^2\} \\
& \le \prod_{s=t}^{t+m} \{1 + w_{s,t}(\sigma_s^2 - \theta) / \tau_t + w_{s,t}^2 (\kappa_s + (\sigma_s^2 - \theta)^2) / \tau_t^2\} \\
& \le \exp\bigg\{\sum_{s=t}^{t+m} \{w_{s,t}(\sigma_s^2 - \theta) / \tau_t + w_{s,t}^2 (\kappa_s + (\sigma_s^2 - \theta)^2) / \tau_t^2\} \bigg\}
\end{aligned}
\end{equation*}
Define $\bar{\sigma_t^2} = \sum_{s=t}^{t+m} w_{s,t} \sigma_s^2 = \sigma_t^2 + \delta_{0,t}$. The RHS can be further bounded by 
\begin{equation*}
\begin{aligned}
\EE\{\exp[r(\theta) / \tau_t]\} & \le \exp\bigg\{(\bar{\sigma_t^2} - \theta) / \tau_t + \sum_{s=t}^{t+m} w_{s,t}^2 (\kappa ^ {\dagger}_t +  2(\sigma_s^2 - {\sigma_t^2})^2 + 2({\sigma_t^2} - \theta)^2) / \tau_t^2\bigg\}  \\
& = \exp\bigg\{(\sigma_t^2 - \theta) / \tau_t + \sum_{s=t}^{t+m} w_{s,t}^2 (\kappa ^ {\dagger}_t + 2(\sigma_t^2 - \theta)^2) / \tau_t^2 + \delta_{0,t}/\tau_t + 2\delta_{1,t} /\tau_t^2 \bigg\}  \\
& = \exp\bigg\{(\sigma_t^2 - \theta) / \tau_t + (n^\dagger_{\eff})^{-1} (\kappa ^ {\dagger}_t + 2(\sigma_t^2 - \theta)^2) / \tau_t^2 + \delta_{0,t}/\tau_t + 2\delta_{1,t} /\tau_t^2 \bigg\}  
\end{aligned}
\end{equation*}
Similarly, we can prove that $\EE\{\exp[-r(\theta) / \tau_t]\} \le \exp\{-(\sigma_t^2 - \theta) / \tau_t + (n^\dagger_{\eff})^{-1} (\kappa ^ {\dagger}_t + 2(\sigma_t^2 - \theta)^2) / \tau_t^2 - \delta_{0,t}/\tau_t + 2\delta_{1,t} /\tau_t^2 \}  $. Define
$$
B_{+}(\theta) = (\sigma_t^2 - \theta) + (\kappa ^ {\dagger}_t + 2(\sigma_t^2 - \theta)^2) / (n^\dagger_{\eff} \tau_t) + \tau_t z
$$
$$
B_{-}(\theta) = (\sigma_t^2 - \theta) - (\kappa ^ {\dagger}_t + 2(\sigma_t^2 - \theta)^2) / (n^\dagger_{\eff} \tau_t)  - \tau_t z
$$
By Chebyshev inequality, 
$$
\PP(r(\theta) > B_{+}(\theta)) \le \exp\{-B_{+}(\theta) / \tau_t\} \EE\exp\{r(\theta) / \tau_t\} = \exp\{-z + |\delta_{0,t}|/\tau_t + 2\delta_{1,t}/\tau_t^2\}
$$
Similarly, $\PP(r(\theta) < B_{-}(\theta)) \le \exp\{-z+ |\delta_{0,t}|/\tau_t + 2\delta_{1,t}/\tau_t^2\}$. Following the same argument with \cite{FLW16}, we can show that for large enough $n^\dagger_{\eff}$ such that $8/(n^\dagger_{\eff} \tau_t) (\kappa ^ {\dagger}_t / (n^\dagger_{\eff} \tau_t) +  \tau_t z) \le 1$, the root $\theta_{+}$ of $B_{+}(\theta)$ satisfies that 
$$
\theta_{+} \le \sigma_t^2 + 2(\kappa ^ {\dagger}_t / (n^\dagger_{\eff} \tau_t) +  \tau_t z)\,,
$$
and the root $\theta_{-}$ of $B_{-}(\theta)$ satisfies that 
$$
\theta_{-} \ge \sigma_t^2 - 2(\kappa ^ {\dagger}_t / (n^\dagger_{\eff} \tau_t) +  \tau_t z)
$$
With the choice of $\tau_t$ given in Theorem \ref{thm1}, we have $\PP(|\hat\sigma_t^2 - \sigma_t^2| \le 4\sqrt{\kappa ^ {\dagger}_t  z/n^\dagger_{\eff}}) \ge 1-2e^{-z+|\delta_{0,t}|/\tau_t + 2\delta_{1,t}/\tau_t^2}$. And the requirement for the effective sample size is that $n^\dagger_{\eff} \ge 16z$.

\end{proof}

\begin{proof}[{\bf Proof of Theorem \ref{thm2}}]
We extend Theorem 2.1 of \cite{wang2020new} to the weighted case. Note again we are solving the following equation for $\hat\tau_t$:
$$
\sum_{s=t}^{t+m}  \frac{\min\bigg(w_{s,t}^2 |X_s^2 - \sigma_t^2|^2, \tau_t^2\bigg)}{\tau_t^2} - z = 0\,.
$$
Also defined $\tau_t$ as the solution of the corresponding population equation:
$$
\sum_{s=t}^{t+m}  \frac{\EE\bigg[ \min\bigg(w_{s,t}^2 |X_s^2 - \sigma_t^2|^2, \tau_t^2\bigg) \bigg]}{\tau_t^2} - z = 0\,.
$$
We will first show that (a) $\tau_t \asymp \sqrt{\frac{\kappa ^ {\dagger}_t}{n^\dagger_{\eff}z}}$ and then (b) with probability approaching $1$, $|\hat\tau_t / \tau_t - 1| \le c_0$ for a small fixed $c_0$. To prove (a), it is straightforward to see that 
$$
\tau_t^2 z \le \sum_{s=t}^{t+m} w_{s,t}^2 \EE|X_s^2 -\sigma_t^2|^2 = \sum_{s=t}^{t+m} w_{s,t}^2 (\kappa_s + (\sigma_s^2 -\sigma_t^2)^2) = \frac{\kappa ^ {\dagger}_t + \delta_{1,t} n^\dagger_{\eff}}{n^\dagger_{\eff}} \le \frac{(1+c_1)\kappa ^ {\dagger}_t}{n^\dagger_{\eff}}
$$
Furthermore, 
\begin{equation*}
\begin{aligned}
\tau_t^2 z &= \sum_{s=t}^{t+m}   \EE\bigg[ (w_{s,t}^2 |X_s^2 - \sigma_t^2|^2) \wedge \tau_t^2 \bigg] \ge  \sum_{s=t}^{t+m} \tau_t^2  \PP\bigg( w_{s,t}^2 |X_s^2 - \sigma_t^2|^2 >\tau_t^2 \bigg) \,.
\end{aligned}
\end{equation*}
Therefore, $z \ge \sum_{s=t}^{t+m} \PP( w_{s,t}^2 |X_s^2 - \sigma_t^2|^2 >\tau_t^2)$. Consider the solution $q_{s}(a)$ to the equation $\PP(|X_s^2 - \sigma_t^2|^2 > q a) = a^{-1}$ of the variable $q$. Note that the solution is unique. Since all $w_{s,t}$ are on the same order, we know that $w_{s,t} \asymp 1/m$, $n^\dagger_{\eff} \asymp m$.
Let $a = a_s = (\sum_{s=t}^{t+m} w_{s,t}^2) / (z w_{s,t}^2 ) \asymp m/z$, the corresponding solution is $q_s(a_s)$. Define $q_{\min} = \min_s q_s(a_s)$. So we have 
$$
\PP\bigg[ w_{s,t}^2 |X_s^2 - \sigma_t^2|^2 > \frac{q_{\min}}{z} \bigg(\sum_{s=t}^{t+m} w_{s,t}^2\bigg) \bigg] \ge \PP\bigg[ |X_s^2 - \sigma_t^2|^2 > q_s(a_s) \bigg(\sum_{s=t}^{t+m} w_{s,t}^2\bigg) /  (z w_{s,t}^2) \bigg] = \frac{z w_{s,t}^2}{\sum_{s=t}^{t+m} w_{s,t}^2}\,.
$$
Let $\tau_0^2 = \frac{q_{\min}}{z} (\sum_{s=t}^{t+m} w_{s,t}^2) = \frac{q_{\min}}{n^\dagger_{\eff} z}$, so we have showed that $\sum_{s=t}^{t+m} \PP( w_{s,t}^2 |X_s^2 - \sigma_t^2|^2 >\tau_0^2) \ge z$. Therefore, $\tau_t \ge \tau_0$. 
From $\PP(|X_s^2 - \sigma_t^2|^2 > q a) = a^{-1}$, we know that for any $s$, $q_s(a_s) \asymp a^{-1} \asymp z/m$, so is $q_{\min}$. Therefore, we have $\tau_t / w_{s,t} \ge \tau_0 / w_{s,t} \asymp m \tau_0 \asymp \sqrt{q_{\min} m/z} \asymp 1$. Write $\tau_t / w_{s,t} \ge c_2$ for some $c_2 \ge 0$.
$$
\tau_t^2 z = \sum_{s=t}^{t+m}   \EE\bigg[ (w_{s,t}^2 |X_s^2 - \sigma_t^2|^2) \wedge \tau_t^2 \bigg] \ge \sum_{s=t}^{t+m} w_{s,t}^2 \EE\bigg[ (|X_s^2 - \sigma_t^2|^2) \wedge c_2^2 \bigg] \asymp \frac{\kappa ^ {\dagger}_t}{n^\dagger_{\eff}}\,.
$$
So we have shown (a) holds, that is, $\tau_t \asymp \sqrt{\frac{\kappa ^ {\dagger}_t}{n^\dagger_{\eff}z}}$.

Next, we need to show (b), so that the solution $\hat\tau_t$ from the second equation gives us the desired optimal truncation rate. To this end, we still follow the proof of Theorem 1 of \cite{wang2020new} closely. Specifically, define $Y_s = w_{s,t} |X_s^2 - \sigma_t^2|$, using their notations, we define
$$
p_n(t) = \sum_s \frac{Y_s^2 I(Y_s \le t)}{t^2}\,, \quad q_n(t) = \sum_s \frac{Y_s^2 \wedge t^2}{t^2} \,, 
$$
and their population versions
$$
p(t) = \sum_s \frac{E[Y_s^2 I(Y_s \le t)]}{t^2}\,, \quad q(t) = \sum_s \frac{E[Y_s^2 \wedge t^2]}{t^2} \,.
$$
One important fact here is that $q_n'(t) = -2t^{-1} p_n(t)$ and $q'(t) = -2t^{-1} p(t)$, which is key to prove Theorem 1 of \cite{wang2020new}. The only difference of our setting here is that we do not assume $Y_s$'s are identically distributed. So when applying Bernstein's inequality as in (S1.8) of \cite{wang2020new}, we need to use the version for non-identically distributed variables, and also bound the sum of individual variances. Specifically, define $\zeta_s = \frac{Y_s^2 \wedge \tau_t^2}{\tau_t^2}$, we have $0\le \zeta_s\le \min\{1, (Y_s\wedge\tau_t)/\tau_t\}$, and hence $\sum_{s} \EE[\zeta_s^2] \le \sum_{s} \EE\Big[\frac{Y_s^2 \wedge \tau_t^2}{\tau_t^2}\Big] = q(\tau_t) = z$. So we can indeed apply Bernstein's inequality on $\sum_s \zeta_s$. For more details, we refer the interested readers to \cite{wang2020new}. 
\end{proof}

\begin{proof}[{\bf Proof of Theorem \ref{thm3}}] Let $\hat\sigma_t = (\hat\sigma_c)_t$.
\begin{equation*}
\begin{aligned}
\PP\bigg[ &\sum_{t=1}^T(\hat\sigma_t^2 / \sigma_t^2 - 1)q_t \ge y \bigg] = \PP\bigg[\sum_{t=1}^T \frac{q_t}{\sigma_t^2} (X_t^2 \wedge c_t - \EE[X_t^2 \wedge c_t]) + \sum_{t=1}^T \bigg(\frac{\EE[X_t^2 \wedge c_t]}{\sigma_t^2}  - 1\bigg)q_t  \ge y \bigg] \\
& \le \PP\bigg[\sum_{t=1}^T \frac{q_t}{\sigma_t^2} (X_t^2 \wedge c_t - E[X_t^2 \wedge c_t])  \ge y/2 \bigg] + \PP\bigg[ \sum_{t=1}^T \bigg(\frac{E[X_t^2 \wedge c_t]}{\sigma_t^2}  - 1\bigg)q_t  \ge y/2 \bigg]
\end{aligned}
\end{equation*}
To bound the first term, we can apply Bernstein inequality for $\sum_t Y_t$ where $Y_t = \frac{q_t}{\sigma_t^2} (X_t^2 \wedge c_t - E[X_t^2 \wedge c_t])$. Note that $E[Y_t^2] \le \tilde\kappa_t q_t^2 / \sigma_t^4 \le  \tilde\kappa_t Q^2/ \sigma_t^4$ and $|Y_t| \le 2c_t Q / \sigma_t^2$, so we can choose $y = C Q\sqrt{zT}$ to make the first term bounded by $e^{-z}$. 

To bound the second term, note that 
$$
\EE[X_t^2 \wedge c_t] - \sigma_t^2 = \EE[X_t^2 I(X_t^2 > c_t)] + \EE[c_t I(X_t^2 > c_t)] \le  \EE\bigg[X_t^2 \cdot \frac{X_t^2}{c_t}\bigg] + c_t \frac{\EE[X_t^4]}{c_t^2} = \tilde\kappa_t \sqrt{\frac{z}{{\tilde\kappa^{\dagger}_t} T}}.
$$
Here we can also choose $y = CQ\sqrt{zT}$ for a large enough $C$ to make the second probability equal to $0$.
\end{proof}

\begin{lem}\label{lem1}
Let $\{Y_t\}$ be a process such that $Y_t = h_t(X_t, X_{t-1}, \dots)$. Define
$$
\gamma_j := \max_t \|h_t(X_t, X_{t-1}, \dots) - h_t(X_t, \dots, X_{t-j+1}, {X_{t-j}}', X_{t-j-1}, \dots)\|_2
$$
for any $j \ge 0$ where ${X_{t-j}}'$ is an iid copy of $X_{t-j}$ and $\{X_t\}$ are independent random innovations satisfying Assumption \ref{assump1}. Assume $\EE[Y_t] = 0, |Y_t| \le M$ for all $t$ and there exists constant $\rho \in (0,1)$ such that
$$
\|Y_{\cdot}\|_2 := \sup_{k\ge 0} \rho^{-k} \sum_{j=k}^{\infty} \gamma_j < \infty\,.
$$
Also assume $T \ge 4 \vee (\log(\rho^{-1})/2)$. We have for $y > 0$, 
$$
\PP\bigg(\sum_{t=1}^T Y_t \ge y \bigg) \le \exp\bigg\{ -\frac{y^2}{4C_1 (T\|Y_{\cdot}\|_2^2 + M^2) + 2C_2 M (\log T)^2 y} \bigg\}\,,
$$
where $C_1 = 2 \max\{ (e^4-5)/4, [\rho(1-\rho) \log(\rho^{-1})]^{-1}\} \cdot (8 \vee \log(\rho^{-1}))^2$, $C_2 = \max\{(c \log 2)^{-1}, [1\vee(\log(\rho^{-1})/8)]\}$ with $c = [\log(\rho^{-1})/8] \wedge \sqrt{(\log 2) \log(\rho^{-1})/4}$.
\end{lem}

The proof of Lemma \ref{lem1} follows closely with Theorem 2.1 of \cite{zhang2021robust}. The only extension here is that we do not require $X_t$ or even $X_t/\sigma_t$ to be identically distributed, so there is no assumption on the stationarity of the process $Y_t$. However, we requires a stronger assumption on the maximal perturbation of each $Y_t= h_t(X_t, X_{t-1}, \dots)$. The entire proof of \cite{zhang2021robust} will go through with this new definition of $\|Y_{\cdot}\|_2$ and $\gamma_j$. We omit the details of the proof.

\begin{proof}[{\bf Proof of Theorem \ref{thm4}}]
Let $\hat\sigma_t = (\hat\sigma_e)_t$ and define $Y_t = \frac{q_t}{\sigma_t^2} (\hat\sigma_t^2 - \EE[\hat\sigma_t^2])$. It is not hard to see that for $j \le m$,
\begin{equation*}
\begin{aligned}
\gamma_j^2 & = \max_t \EE\bigg\{[(w_{j,0} X_{t+j}^2) \wedge c_t - (w_{j,0} {{X_{t+j}}'}^2) \wedge c_t]^2 \bigg\}  \frac{q_t^2}{\sigma_t^4} \\
& \le \max_t \EE\bigg\{(w_{j,0} X_{t+j}^2 - w_{j,0} {{X_{t+j}}'}^2)^2 + (w_{j,0} X_{t+j}^2 - c_t)^2 I(w_{j,0} X_{t+j}^2 < c_t, w_{j,0} {{X_{t+j}}'}^2 > c_t) \\
& \quad\quad + (w_{j,0} {{X_{t+j}}'}^2  - c_t)^2 I(w_{j,0} X_{t+j}^2 > c_t, w_{j,0} {{X_{t+j}}'}^2 < c_t) \bigg\} \frac{Q^2}{\sigma_t^4}\\
& \le 4 w_{j,0}^2 Q^2  \max_t \tilde \kappa_{t+j} / \sigma_t^4  \le 4 w_{j,0}^2 Q^2  \max_{t,u\le m} \tilde \kappa_{t+u} / \sigma_t^4
\end{aligned}
\end{equation*}
And for $j > m$, $\gamma_j = 0$. Therefore, 
$$
\|Y_{\cdot}\|_2 = \sup_k \rho^{-k} \sum_{j=k}^{\infty} \gamma_j \le 2 Q \sqrt{\max_{t,u\le m} \tilde \kappa_{t+u} / \sigma_t^4}  \sup_k \rho^{-k} \sum_{j=k}^{m} w_{j,0} \le 2 Q \sqrt{\max_{t,u\le m} \tilde \kappa_{t+u} / \sigma_t^4} < \infty\,,
$$
for any fixed $\rho \in (0,1)$. 

In addition, we claim $|Y_t| \le CQ\sqrt{z}$ with high probability. To prove this, we need the following result: when $n^\dagger_{\eff} \ge 16 \tilde z$ and $(n^\dagger_{\eff} c_t)^2 \ge 16{\tilde\kappa^{\dagger}_t}$, 
$$
\PP(|\hat\sigma_t^2 - \sigma_t^2| \le 2({\tilde\kappa^{\dagger}_t} / (n^\dagger_{\eff} c_t) +  c_t \tilde z) \ge 1-2e^{-\tilde z+|\delta_{0,t}|/c_t + 2\delta_{1,t}/c_t^2}\,.
$$
This can be shown following a similar proof as Theorem \ref{thm1}, thus we omit the details.
Note that here $c_t$ is not chosen to optimize the error bound, as we have another average over $T$ to take care of the extra variance in $\hat\sigma_t^2$. So here we only need to choose $c_t$ to make sure the error bound to be in the order of $\sqrt{z}$. $c_t =\sqrt{\frac{{\tilde\kappa^{\dagger}_t}T}{n^{\dagger 2}_{\eff} z}}$ and pick $\tilde z = c z n^\dagger_{\eff} / \sqrt{T}$ can indeed do the job, since the exception probability is $2e^{-\frac{c}{2} n^\dagger_{\eff} z / \sqrt{T}}$ under the assumption that $\max_t |\delta_{0,t}|/\sqrt{{\tilde\kappa^{\dagger}_t}} + 2\delta_{1,t} n^\dagger_{\eff}/({\tilde\kappa^{\dagger}_t}\sqrt{T}) \le c /2$. We require  $|Y_t| \le CQ\sqrt{z}$ to hold for all time points, so the exception probability for all events is bounded by $2 T e^{-\frac{c}{2} n^\dagger_{\eff} z / \sqrt{T}}$. When $n^\dagger_{\eff} > 2 c^{-1} (1 + \log 2T/z) \sqrt{T}$, this is further bounded by $e^{-z}$. 
Finally $n^\dagger_{\eff} \ge 16 \tilde z$ and $(n^\dagger_{\eff} c_t)^2 \ge 16{\tilde\kappa^{\dagger}_t}$ lead to the requirements of $\sqrt{T} \ge 16 c z$ and $T \ge 16z$.
Now conditioning on that $|Y_t| \le CQ\sqrt{z}$, we are ready to call Lemma \ref{lem1} for $Y_t$. We can choose $y = C Q\sqrt{zT}$ in Lemma \ref{lem1} to make the exception probability smaller than $e^{-z}$. So in total the exception probability is $2e^{-z}$.

Next, in terms of the bias term $\EE[\hat\sigma_t^2] - \sigma_t^2 = \sum_{s=t}^{t+m} w_{s,t} \EE[ \min( X_s^2, c_t/w_{s,t}) - \sigma_s^2] +  (\sum_{s=t}^{t+m} w_{s,t} \sigma_s^2 - \sigma_t^2)$. From the proof of Theorem \ref{thm3}, we know that $\EE[ \min( X_s^2, c_t/w_{s,t}) - \sigma_s^2] \le 2 w_{s,t} \tilde\kappa_s / c_t$. Therefore
\begin{equation*}
\begin{aligned}
\frac{\EE[\hat\sigma_t^2]}{\sigma_t^2} - 1 \le \frac{2}{\sigma_t^2} \sum_{s=t}^{t+m} w_{s,t}^2 \tilde\kappa_s / c_t + \frac{\delta_{0,t}}{\sigma_t^2} = 2 \sqrt{\frac{{\tilde\kappa^{\dagger}_t} z}{\sigma_t^4 T}} + \frac{\delta_{0,t}}{\sigma_t^2}.
\end{aligned}
\end{equation*}
Thus, the bias term will not affect the total error bound as shown in the theorem. 
\end{proof}

\begin{lem} \label{lem2}
% Denote $F(\theta)$ as the corresponding Huber loss,
% $$
% F(\theta) = \sum_{s=t}^{t+m} w_{s,t} \ell_s(\theta)\,, \text{ where } \ell_s(\theta) = \begin{cases}
%         \frac{c_t}{w_{s,t}}(X_s^2 - \theta) - \frac{c_t^2}{2w_{s,t}^2}, \quad \text{   if $X_s^2 - \theta > c_t / w_{s,t}$,}
%         \\
%         \frac{1}{2} (X_s^2 - \theta)^2, \quad\quad\quad\;\;\; \text{   if $|X_s^2 - \theta| \le c_t / w_{s,t}$,}
%         \\
%       - \frac{c_t}{w_{s,t}}(X_s^2 - \theta) - \frac{c_t^2}{2w_{s,t}^2}, \;  \text{  if $X_s^2 - \theta < -c_t / w_{s,t}$.}
%         \end{cases}
% $$
Assume the weighted Huber loss $F(\theta) = \cL_{c_t}(\theta; \{w_{s, t}\}_{s = t} ^ {t + m})$ is $\alpha$-strongly convex for some $\alpha \in (0,1/2)$ in some local neighborhood $\Theta$ around $\theta^* \in \Theta$. $\tilde F(\theta)$ is the perturbed version of $F(\theta)$. If $\theta^* = \argmin_{\theta} F(\theta)$ and $\tilde\theta^* = \argmin_{\theta} \tilde F(\theta)$ and $\tilde\theta^* \in \Theta$, we have 
$$
|\tilde\theta - \theta| \le \frac{\alpha+1}{\alpha} \sup_\theta | \tilde F'(\theta) - F'(\theta)|\,.
$$
\end{lem}

\begin{proof}[{\bf Proof of Lemma \ref{lem2}}]
Besides strong convexity, we know that $F(\theta)$ is $\beta$-smooth with $\beta = 1$. That is 
$$
F(\theta_1) - F(\theta_2) - F'(\theta_2) (\theta_1 - \theta_2) \le \frac{\beta}{2} (\theta_1 - \theta_2)^2.
$$
$\beta = 1$ is obvious given the second derivative of $F(\theta)$ is bounded by 1. From Lemma 3.11 of \cite{bubeck2014convex}, for $\theta_1,\theta_2 \in \Theta$,
$$
\frac{\alpha}{\alpha+1} (\theta_1 - \theta_2)^2 \le (F'(\theta_1) - F'(\theta_2))(\theta_1 - \theta_2) \le \frac{\alpha}{2(\alpha+1)} (\theta_1 - \theta_2)^2 + \frac{\alpha+1}{2\alpha} (F'(\theta_1) - F'(\theta_2))^2.
$$
Choose $\theta_1 = \tilde\theta^*, \theta_2 = \theta^*$, so $F'(\theta_2) = 0 = \tilde F'(\theta_1)$. Then
$$
\frac{\alpha}{\alpha+1} (\tilde\theta^* - \theta^*)^2 \le  \frac{\alpha+1}{\alpha} (F'(\tilde\theta^*) - \tilde F'(\tilde\theta^*))^2,
$$
which concludes the proof.
\end{proof}

\begin{proof}[{\bf Proof of Theorem \ref{thm5}}]
Let $\hat\sigma_t = (\hat\sigma_H)_t$ and define $Y_t = \frac{q_t}{\sigma_t^2} (\hat\sigma_t^2 - \EE[\hat\sigma_t^2])$. In order to apply Lemma \ref{lem1}, we need to employ Lemma \ref{lem2} to bound the perturbation of Huber loss minimizer via bounding the perturbation of Huber loss derivative. Similar to the proof of Theorem \ref{thm4}, we can show that $|\hat\sigma_t^2 - \sigma_t^2| \le C\sqrt{z}$ for all $t$ with probability larger than $1-e^{-z}$. 
We can actually explicitly write out the bound for $|\hat\sigma_t^2 - \sigma_t^2|$ following the proof of Theorem \ref{thm4}: $|\hat\sigma_t^2 - \sigma_t^2| \le 2(\sqrt{\kappa_t^{\dagger} z/T} + c\sqrt{\kappa_t^{\dagger} z}) < (2c+2)\sqrt{\kappa_t^{\dagger} z}$, which means the Huber loss minimizer does fall into the region we have strong convexity by our assumptions in Theorem \ref{thm5}.
The bound on $|\hat\sigma_t^2 - \sigma_t^2|$ also implies that $|Y_t|\le CQ\sqrt{z}$ with exception probability of $e^{-z}$. 

In addition, we would like to check $\|Y_{\cdot}\|_2 < \infty$. For $j > m$, $\gamma_j = 0$ and for $j \le m$,
\begin{equation*}
\begin{aligned}
\gamma_j^2 & \le \frac{(\alpha+1)^2}{\alpha^2} \max_t \max_{\sigma_t^2} \EE\bigg\{\bigg[(w_{j,0} |X_{t+j}^2-\sigma_t^2|) \wedge c_t - (w_{j,0} |{{X_{t+j}}'}^2-\sigma_t^2|) \wedge c_t\bigg]^2 \bigg\}  \frac{q_t^2}{\sigma_t^4} \\
& \le \frac{(\alpha+1)^2}{\alpha^2}  \max_t \max_{\sigma_t^2} \EE\bigg\{(w_{j,0} |X_{t+j}^2-\sigma_t^2| - w_{j,0} |{{X_{t+j}}'}^2-\sigma_t^2|)^2  \\
& \quad\quad + (w_{j,0} |X_{t+j}^2-\sigma_t^2| - c_t)^2 I(w_{j,0} |X_{t+j}^2-\sigma_t^2| < c_t, w_{j,0} |{{X_{t+j}}'}^2-\sigma_t^2| > c_t) \\
& \quad\quad + (w_{j,0} |{{X_{t+j}}'}^2-\sigma_t^2|  - c_t)^2 I(w_{j,0} |X_{t+j}^2-\sigma_t^2| > c_t, w_{j,0} |{{X_{t+j}}'}^2-\sigma_t^2| < c_t) \bigg\} \frac{Q^2}{\sigma_t^4}\\
& \le 4 w_{j,0}^2 Q^2  \max_t \kappa_{t+j} / \sigma_t^4 \le 4 w_{j,0}^2 Q^2  \max_{t, u\le m} \kappa_{t+u} / \sigma_t^4 \,.
\end{aligned}
\end{equation*}
Therefore $\|Y_{\cdot}\|_2 < \infty$ for any fixed $\rho \in (0,1)$. We can indeed apply Lemma \ref{lem1} to bound the sum of $Y_t$, which is of order $CQ\sqrt{zT}$, with exceptional probability of another $e^{-z}$.

Finally, we bound the bias term $\EE[\hat\sigma_t^2] / \sigma_t^2 - 1$. Note that
\begin{equation*}
\begin{aligned}
0 & = \sum_{s=t}^{t+m} w_{s,t} \EE\bigg[\min\bigg(|X_s^2 - \hat\sigma_t^2|, \frac{c_t}{w_{s,t}}\bigg) \sgn(X_s^2 - \hat\sigma_t^2)\bigg] \\
& =  \sum_{s=t}^{t+m} w_{s,t} \EE\bigg[(X_s^2 - \hat\sigma_t^2) I\bigg(|X_s^2 - \hat\sigma_t^2| \le \frac{c_t}{w_{s,t}}\bigg) + \frac{c_t}{w_{s,t}}  I\bigg(|X_s^2 - \hat\sigma_t^2| > \frac{c_t}{w_{s,t}}\bigg) \bigg] \\
& =\delta_{0,t} + \EE[(X_t^2 - \hat\sigma_t^2)] + \sum_{s=t}^{t+m} w_{s,t} \EE\bigg[(X_s^2 - \hat\sigma_t^2) I\bigg(|X_s^2 - \hat\sigma_t^2| > \frac{c_t}{w_{s,t}}\bigg) + \frac{c_t}{w_{s,t}}  I\bigg(|X_s^2 - \hat\sigma_t^2| > \frac{c_t}{w_{s,t}}\bigg) \bigg] 
\end{aligned}
\end{equation*}
Similar to the proof of Theorem \ref{thm3}, we know that the $s$-th component of the third term can be bounded by $2 w_{s,t} \EE[(X_s^2 - \hat\sigma_t^2)^2] / c_t \le 2 w_{s,t} (2\kappa_s + 2(\sigma_s^2 - \sigma_t^2)^2+ 2 \EE[(\sigma_t^2 - \hat\sigma_t^2)^2]) / c_t$. Therefore
\begin{equation*}
\begin{aligned}
\frac{\EE[\hat\sigma_t^2]}{\sigma_t^2} - 1 & \le \frac{4}{\sigma_t^2} \sum_{s=t}^{t+m} w_{s,t}^2 (\kappa_s + \EE[(\sigma_t^2 - \hat\sigma_t^2)^2])  / c_t + \frac{\delta_{0,t} + 4\delta_{1,t}}{\sigma_t^2} \\
& = 4 \sqrt{\frac{{\kappa^{\dagger}_t} z}{\sigma_t^4 T}} + \frac{\delta_{0,t} + 4\delta_{1,t}}{\sigma_t^2} + \frac{4 E[(\sigma_t^2 - \hat\sigma_t^2)^2]}{ \sigma_t^4} \sqrt{\frac{\sigma_t^4 z}{{\kappa^{\dagger}_t} T}}.
\end{aligned}
\end{equation*}
Furthermore, it is not hard to show that $E[(\sigma_t^2 - \hat\sigma_t^2)^2]$ is bounded using Lemma \ref{lem2}. Therefore, the bias is indeed of the order $\sqrt{z/T}$ with the additional approximation error rate $\frac{\delta_{0,t} + 4\delta_{1,t}}{\sigma_t^2} $. The proof is now complete.
\end{proof}

\begin{proof}[{\bf Proof of Theorem \ref{thm6}}]
Following the same proof as Theorem \ref{thm1}, we have that
$$
\PP(|h_t- \sigma_t^2| \le 4 \sqrt{\kappa ^ {\ddagger}_t z/n^\ddagger_{\eff}}) \ge 1-2e^{-z+ |\Delta_{0,t}| \sqrt{{n^\ddagger_{\eff} z}/{\kappa ^ {\ddagger}_t}} + 2\Delta_{1,t}{n^\ddagger_{\eff} z}/{\kappa ^ {\ddagger}_t} }
$$
Conditioning on this event, $L$ is Lipchitz in the second argument. Then the error bound is enlarged by $B$ times. 
\end{proof}

\begin{proof}[{\bf Proof of Theorem \ref{thm7}}]
Recall that
\begin{equation*}
\begin{aligned}
& \text{\rnum{2}} \le \PP\bigg(\bigg|\frac{1}{T} \sum_{t=1}^T (L(\sigma_t^2, h_t) - \EE[L(\sigma_t^2, h_t)]) \bigg| > \frac{\varepsilon}{2}\bigg) + \PP\bigg(\bigg| \frac{1}{T} \sum_{t=1}^T (C(h_{t}) - \EE[C(h_{t})])(\hat\sigma_t^2 - \sigma_t^2) \bigg| > \frac{\varepsilon}{2} \bigg) \\
& =: \Delta_A + \Delta_B\,.
\end{aligned}
\end{equation*}

\textit{\underline{Now let us prove (i) first. }}
We first bound $\Delta_A$. We still apply Lemma \ref{lem1} to bound the concentration. Let $Y_t = L(\sigma_t^2, h_t) - \EE[L(\sigma_t^2, h_t)]$. From the proof of Theorem \ref{thm6}, we know that 
$$
\PP(|h_t- \sigma_t^2| \le 4 \sqrt{\kappa ^ {\ddagger}_t z/n^\ddagger_{\eff}}) \ge 1-2e^{-z/2}\,.
$$
So $|h_t- \sigma_t^2| \le 4 \sqrt{\kappa ^ {\ddagger}_t z/n^\ddagger_{\eff}} \le \sigma_t^2/2$ with exception probability of $2e^{-z/2}$. Applying the union bound, we get $|h_t- \sigma_t^2| \le \sigma_t^2/2$ for all $t$ with probability $\ge 1-2Te^{-z/2}$. On this event, we have $|Y_t| \le C$ because $|L(\sigma_t^2, h_t) - L(\sigma_t^2, \sigma_t^2)| \le B(\sigma_t^2/2) |h_t - \sigma_t^2| \le B(\sigma_t^2/2) \sigma_t^2/2 < \infty$. Similar to previous proofs, except that now we look data backward, for $j \ge 0$,
\begin{equation*}
\begin{aligned}
\gamma_j^2 & = \max_t \EE\bigg\{ [L(\sigma_t^2, h_t(X_{t-1}, \dots, X_{t-m})) -  L(\sigma_t^2, h_t(X_{t-1}, \dots, {X_{t-j-1}}', \dots, X_{t-m}))]^2 \bigg\} \\
& \le \max_t B^2(\sigma_t^2/2) \max_t \EE\bigg\{ [h_t(X_{t-1}, \dots, X_{t-m}) -  h_t(X_{t-1}, \dots, {X_{t-j-1}}', \dots, X_{t-m})]^2 \bigg\}\\
& \le 4 \nu_{-j,1}^2  \max_t B^2(\sigma_t^2/2) \frac{(\alpha+1)^2}{\alpha^2} \max_t \kappa_{t}
\end{aligned}
\end{equation*}
The last inequality can be shown similar to the proof of Theorem \ref{thm5} with the assumption that the Huber loss is $\alpha$-strongly convex locally. Therefore, $\|Y_{\cdot}\|_2 < \infty$ for any fixed $\rho \in (0,1)$. We apply Lemma \ref{lem1} on $Y_t$ and again pick $y = C\sqrt{zT}$ to make the exceptional probability $2e^{-z}$. So we get $\PP(| \Delta_A | \le C\sqrt{z/T}) \ge 1 - 2(T+1) e^{-z}$.

Now let us try to apply Lemma \ref{lem1} on $\Delta_B$. Let $Y_t = (C(h_{t}) - \EE[C(h_{t})])(\hat\sigma_t^2 - \sigma_t^2)$. Note that since $C(\cdot)$ is non-increasing, so $C(h_t) \le C(\sigma_t^2 / 2)$ is bounded. 
If we use the second and third proxies, in the proofs of Theorem \ref{thm4} and \ref{thm5}, we have showed that $|\hat\sigma_t^2 - \sigma_t^2| \le C\sqrt{z}$ for all $t$ with exception probability at most $e^{-z}$. Therefore, we conclude $|Y_t| \le C\sqrt{z}$ for all $t$ with exception probability at most $e^{-z}$.  
%If we use the first proxy, according to Theorem \ref{thm3}, we know $|\hat\sigma_t^2 - \sigma_t^2| \le c_t + \sigma_t^2 \le C\sqrt{T/z}$, which is less than $C\sqrt{zT}/(\log T)^2$ if $z \ge c(\log T)^2$.
Now for bounding $\gamma_j$, note that $Y_t$ actually is a function of $X_{t-m}, \dots, X_{t-1}, X_t, \dots, X_{t+m}$ with the first $m$ data constructing predictors and the remaining $m+1$ data constructing proxies. Hence, for $j < m$, it is not hard to show $\gamma_j^2 \le C \nu_{-(m-j),1}^2$ for some $C >0$ and for $m \le j \le 2m$, $\gamma_j^2 \le C w_{j-m, 0}^2$. So we have $\|Y_{\cdot}\|_2 \le 2\sqrt{C} < \infty$. Applying Lemma \ref{lem1} will again give us  $\PP(| \Delta_B | \le C\sqrt{z/T}) \ge 1 - 3 e^{-z}$.

Combining results for $\Delta_A$ and $\Delta_B$ and choose $\varepsilon = C\sqrt{z/T}$ for large enough $C$, we conclude (i) for bounding \rnum{2}.

\medskip
\textit{\underline{Next we prove (ii) and (iii).}} The proof follows the exact same arguments as (i) except for a few bounding details. 

Firstly, in bounding $\Delta_A$, we need $L(\sigma_t^2, h_t)$ to be bounded. In (iii), $h_t$ is not necessarily bounded, but we directly work with a bounded loss $L(\sigma_t^2, h_t) \le M_0$. In (ii), $h_t \le M$ and we claim $h_t \ge \sigma_t^2/2$ with probability $\ge 1-2e^{-z}$, thus $L(\sigma_t^2, h_t) \le B_t(M) M$. To see why the claim holds, define $\check h_t = \sum_{s=t-m}^{t-1} \nu_{s,t} \min(X_s^2, \tau_t / \nu_{s,t})$. The robust predictor proposed in Section \ref{sec4.1} achieves central fourth moment, while $\check h_t$ can achieve the same rate of convergence with absolute fourth moment. This is similar to the difference between the second and third proxy option. Similar to the proof of Theorem \ref{thm6}, we can show 
$$
\PP\bigg(|\check h_t- \sigma_t^2| \le 4 \sqrt{{\tilde\kappa^{\ddagger}_t} z/n^\ddagger_{\eff}} \le \sigma_t^2/2\bigg) \ge 1-2e^{-z/2}\,.
$$
So we know that $\check h_t \ge \sigma_t^2/2$ with probability $\ge 1-2e^{-z}$. Interestingly, $h_t \ge \check h_t$. So we always have good control of the left tail due to the positiveness of squared data. 

Secondly, in bounding $\Delta_A$, we need $\gamma_j^2 \le C \nu_{-j,1}^2$. In (iii), since loss is Lipchitz in the whole region, we can easily see that $\gamma_j^2 \le 4 \nu_{-j,1}^2 B_0^2 \max_t \tilde\kappa_t$. In (ii), loss is Lipchitz in the local region of $\sigma_t^2/2 \le h_t \le M$, but we know with high probability the clipped predictor indeed falls into the region, so we have $\gamma_j^2 \le 4 \nu_{-j,1}^2 \max_t B_t^2(M) \max_t \tilde\kappa_t$. Therefore, we have no problem bounding $\Delta_A$.

Thirdly, in bounding $\Delta_B$, we require $C(h_t)$ to be bounded. Note that since $h_t \ge \sigma_t^2$ with high probability, we have $C(h_t) \le C(\sigma_t^2 / 2)$ even when we use non-robust predictors in (ii) and (iii). So we can bound $\Delta_B$ as desired too. 

Finally putting everything together and choose $\varepsilon = C\sqrt{z/T}$ for large enough $C$, we conclude (ii) and (iii) for bounding \rnum{2}.

\end{proof}

\end{document}